\documentclass[11pt]{article}

\usepackage[a4paper,hmargin=1in,vmargin=1.25in]{geometry}
\usepackage[utf8]{inputenc}
\usepackage[toc,page]{appendix}
\usepackage{multicol}
\usepackage{amsmath,amsfonts,amssymb,amsthm,url}
\usepackage{tikz}
\usepackage{longtable}
\usepackage{booktabs,authblk}
\usepackage{multirow}
\usepackage{threeparttable}
\usetikzlibrary{shapes,arrows}
\usetikzlibrary{backgrounds}

\usepackage{hyperref}

\def\sss{\smallsetminus}

\newcommand*{\lon}{
       \mskip1mu
        \relax
        {;}
        \mskip1mu
        \relax
}

\usepackage{caption}
\usepackage{subcaption}
\usepackage{mathtools}

\newcommand\fullv[1]{#1}
\newcommand\confv[1]{}

\def\Real{\mathbf{R}} 

\def\Par{\mathcal{P}} 

\def\cihy{\mathcal{C}_o} 

\def\Mat{M} 

\def\field{\mathbf{F}}
\def\kos{K}
\def\kosl{L}

\def\cip{\texttt{CI}}

\def\gep{\texttt{GE}}
\def\cep{\texttt{CE}}
\def\imp{\texttt{IM}}
\def\akp{\texttt{AK}}
\def\clp{\texttt{CL}} 

\def\vamos{V} 
\def\vamosp{V_o} 

\def\pone{P_{8,1}}
\def\ptwo{P'_{8,2}}
\def\ptwop{P''_{8,2}}
\def\pthree{P_{8,3}}

\def\ttt{T^3} 
\def\tto{T^3_1} 
\def\ttw{T^3_2} 
\def\ttc{T^m_1} 
\def\twc{T^m_2} 

\def\sss{\smallsetminus}

\def\grse{E} 

\theoremstyle{plain}
\newtheorem{proposition}{Proposition}[section]

\newtheorem{lemma}[proposition]{Lemma}

\theoremstyle{definition}
\newtheorem{definition}[proposition]{Definition}

\title{A Note on Extension Properties and Representations of Matroids}
\author[1]{Michael Bamiloshin}
\affil[1]{Universitat Rovira i Virgili, Tarragona, Spain}
\author[1]{Oriol Farr\`as}
\author[2]{Carles Padr\'o}
\affil[2]{Universitat Polit\`ecnica de Catalunya, Barcelona, Spain\\
\texttt{\small{michaelolugbenga.bamiloshin@urv.cat, oriol.farras@urv.cat, carles.padro@upc.edu}}}

\begin{document}
\maketitle

\begin{abstract}
We discuss several extension properties 
of matroids and polymatroids and 
their application as necessary conditions
for the existence of different matroid representations,
namely linear, folded linear, algebraic, and entropic representations.
Iterations of those extension properties are 
checked for  matroids on eight and nine elements
by means of computer-aided explorations,
finding in that way several new examples
of non-linearly representable matroids.
A special emphasis is made on sparse paving matroids
on nine points containing the tic-tac-toe configuration.
We present a new, more clear description of
that family and we analyze
extension properties on those matroids and their duals. 
 
 \
 
\end{abstract}

\makeatletter{\renewcommand*{\@makefnmark}{}
	\footnotetext{Michael Bamiloshin received funding from the European Union's Horizon 2020 research and innovation programme under the Marie Sk\l{}odowska-Curie grant agreement No. 713679 and from the Universitat Rovira i Virgili.  Oriol Farr\`as is supported by grant 2021 SGR 00115 from the Government of Catalonia and by then project HERMES, funded by INCIBE and by the European Union NextGeneration EU/PRTR. Carles Padr\'o is supported by the Spanish Government through grant PID2019-109379RB-I00. Additionally, the authors are supported by the project ACITHEC PID2021-124928NB-I00, funded by MCIN/AEI/10.13039/ 501100011033/FEDER, EU. }\makeatother}

\section{Introduction}
\label{sec:introd}


Given a linear representation of a matroid $\Mat$,
every flat determines a vector subspace.
Consider a non-modular pair $(F_1,F_2)$ of flats of $\Mat$, 
that is, the rank of $F_1 \cap F_2$ is less than the dimension
of the intersection of the corresponding subspaces.
Then the rank of $F_1 \cap F_2$
is increased in some single-element extension of $\Mat$
that is still linearly representable. 
A necessary condition for a matroid
to be linearly representable
is derived from that fact, namely the
\emph{generalized Euclidean intersection property} discussed in~\cite{BaWa89}.
In particular, it is not satisfied by the \emph{V\'amos matroid}.

A necessary condition for a matroid
to be algebraic is given by the
Ingleton-Main lemma~\cite{InMa75}.
It states that, given three non-coplanar lines 
in an algebraic matroid such that 
every two of them are coplanar, there
is a single-element extension in which 
the three lines meet in one point and, 
moreover, the extension is algebraic.
More general versions of that property 
were presented by Lindstr\"om~\cite{Lindstrom1988}
and Dress and Lov\'asz~\cite{DrLo87}.
The Ingleton-Main lemma provided the first example of a
non-algebraic matroid~\cite{InMa75}, namely the V\'amos matroid.
Lindstr\"om~\cite{Lindstrom1988} used a generalization of it
to prove that the class of algebraic matroids 
has infinitely many excluded minors.
Hochst\"attler~\cite{Hoc97} proved that 
the dual of the \emph{tic-tac-toe matroid}
is not algebraic by using the
Ingleton-Main lemma.
That could be a counterexample
proving that the class of algebraic matroids
is not closed by duality, but this is still an open problem. 

Using the terminology introduced in~\cite{Bollen18},
those necessary conditions for a matroid to be 
linear or algebraic are examples
of \emph{extension properties} of matroids. 
Each of them consists of a matroid extension
with certain constraints.
Every matroid in the class of interest
admits extensions with the required properties
that are in the same class.

Extension properties of
polymatroids have been used in information theory.
Specifically, the Ahlswede-K\"orner lemma~\cite{AhKo77,AhKo06,Kac13}
and the copy lemma~\cite{DFZ06,Kac13,ZhYe98}
determine extension properties of 
\emph{almost entropic polymatroids}. These properties
have been applied in the search
for \emph{non-Shannon linear information
inequalities}~\cite{DFZ06,DFZ11,HRSV00,Kac13,Mat07-3},
with the \emph{Zhang-Yeung inequality}~\cite{ZhYe98}
being the first one to be found.
\emph{Non-Shannon linear rank inequalities}
as, for example, the \emph{Ingleton inequality}~\cite{Ing71}
are useful when dealing with discrete 
\emph{linear} random variables. 
Almost all known such inequalities 
follow from  the \emph{common information property}~\cite{DFZ09},
an extension property of linearly representable polymatroids.

Linear information and rank inequalities
have been used as constraints in
linear programs providing bounds 
on the information ratio of
\emph{secret sharing schemes}~\cite{BLP08,BeOr11,PVY13,MPY16}
and on the achievable rates in \emph{network coding}~\cite{DFZ07,TCG17,Yeu08}. 
Some of those bounds have been recently
improved by using the aforementioned
extension properties of polymatroids,
from which information and rank 
inequalities are derived~\cite{FKMP20,GuRo19,BBFP20}.
The idea is to substitute those inequalities  
in the linear programs by constraints
that are derived from extension properties.
In that way, it is not neccessary to guess what
inequalities fit best with a given problem. 
Among other results, those improved bounds made it possible to
determine the optimal value of the information ratio
of \emph{linear} secret sharing schemes
for all access structures on five players
and all graph-based access structures on six players~\cite{FKMP20},
partially concluding the projects initiated in~\cite{Dij97,JaMa96}.

Frobenius flocks, 
which have been introduced by Bollen~\cite{Bollen18},
provide a remarkable new tool to find out the existence of algebraic
representations of matroids.
Exhaustive searches on small matroids
have been carried out in~\cite{Bollen18}
by using Frobenius flocks as well as
Ingleton-Main and Dress-Lov\'asz extension properties.

Building on the results by Mayhew and Royle~\cite{MaRo08}
and their online database of matroids~\cite{RoMaDatabase}
and using, among other tools, 
common information and Ahslwede-K\"orner
extension properties,
the classification of matroids 
on small ground sets according to their representations
was pursued in~\cite{BBFP20}.
In addition to linear and algebraic, 
almost entropic and \emph{folded linear} representations
were considered. 
Those two classes of matroids 
play an important role in 
the theory of secret sharing schemes
because of the Brickell-Davenport theorem~\cite{BrDa91}.
Moreover, the interest of the former is
increased by a recent result by Mat\'u\v{s}~\cite{Matus24},
namely, algebraic matroids are almost entropic.
Therefore, extension properties of 
almost entropic polymatroids
apply to algebraic matroids too.

In contrast to previous 
works~\cite{InMa75,Lindstrom1988,DrLo87,BaWa89,Hoc97},
more recent applications of 
extension properties~\cite{Bollen18,FKMP20,GuRo19,BBFP20}
have been obtained by computer-aided explorations 
in which several iterations
of the chosen extension are searched.
On the negative side, due to their computational complexity, 
those explorations are only feasible for
matroids and polymatroids on small ground sets. 

The classification of matroids on eight elements
is almost concluded by the results 
in~\cite{MaRo08,Bollen18,BBFP20}.
Only for three of them 
it is not known whether they are  
algebraic, almost entropic, or neither.
There are exactly 39 
matroids on eight elements that do not 
satisfy the Ingleton inequality~\cite{MaRo08}.
They are \emph{sparse paving} matroids 
and each of them contains five \emph{circuit-hyperplanes} 
in the same configuration as the ones of the V\'amos matroid.
Because of that, those matroids 
do not satisfy the Ahlswede-K\"orner property,
and hence they are neither algebraic nor almost entropic.
Moreover, all of them are relaxations of 
the \emph{maximal} sparse paving matroid $AG(3,2)$.

One of the main open problems about
the classification of matroids on nine elements
is to determine whether the 
tic-tac-toe matroid $\ttt$ is algebraic or not.
Its dual matroid does not satisfy the 
Ingleton-Main extension property~\cite{Hoc97}, and hence it is not algebraic.
Therefore, $\ttt$ is a candidate for a  counterexample
proving that the class of algebraic matroids 
is not closed by duality.

\subsection{Our Results}


In this work, we pursue the application of
extension properties to the classification of 
small matroids according to their representations.
 We provide a unified view of extension properties 
previously studied in the area of Matroid Theory 
(i.e., Euclidean, generalized Euclidean, 
Levi's intersection, and Ingleton-Main) 
with others studied in the area of Information Theory 
(common information and Ahlswede-K\"orner) 
that help us combine theoretical results and
computer-aided tools for the classification
of representable matroids. 
We discover a new family of matroids connected to 
the tic-tac-toe matroid that are not representable. 
We mainly focus on sparse paving matroids on 
eight and nine elements.

In the first stage of this work, we carried out
computer-aided explorations
on the databases of matroids
by Royle and Mayhew~\cite{RoMaDatabase} and
Bollen~\cite{Bollen18-1}.
Using linear programs in a blanket approach to
check the existence of iterated extension properties
for matroids on nine points
appeared to be too computationally costly.
Instead, we checked the feasibility of 
iterated generalized Euclidean extensions
by exploring the existence of the
associated modular cuts~\cite[Section~7.2]{Oxley11}. Once impossible iterated generalized Euclidean extensions are found,
one can check by linear programming
other extension properties
on those particular situations,
like common information and 
Ahlswede-K\"orner. 
That strategy provided several 
new examples of non-linear matroids,
which are listed in the \fullv{Appendix}\confv{extended 
version of this paper~\cite{BFP23}}.
The computer programs 
are available in~\cite{Bam22}.


The outcome of those explorations
lead us to analyze in more detail 
the sparse paving matroids that are
relaxations of~$P_8$ and the ones 
that present the tic-tac-toe configuration.

There are only five matroids on eight elements
that satisfy the Ingleton inequality but are not linear.
Two of them are folded linear~\cite{BBFP20} 
and algebraic~\cite{Lindstrom1986,BCD18}.
The other three, which are relaxations of 
the maximal sparse paving matroid $P_8$,
are not folded linear~\cite{BBFP20} and it is not known
whether they are algebraic, almost entropic, or neither.
We 
checked by computer-aided explorations
that they do not satisfy the common information property.
Therefore, the same applies to the generalized Euclidean property.
For two of those matroids, we present a 
human readable proof for that fact.
In particular, the V\'amos configuration appears
after some generalized Euclidean extensions, 
which may be a hint indicating 
that those matroids are not algebraic.

In addition to the configurations of the 
V\'amos matroid and the aforementioned relaxations of~$P_8$,
the configurations of the circuit-hyperplanes of
the tic-tac-toe matroid $\ttt$ 
and its dual are also of interest
in regard to extension properties.
The matroid $\ttt$ does not satisfy the 
generalized Euclidean property~\cite{AlHo95}, 
but it satisfies the Ingleton-Main 
and Dress-Lov\'asz properties~\cite{Bollen18}.
In contrast, its dual matroid
does not satisfy the Ingleton-Main property~\cite{Hoc97}. Similarly to the V\'amos configuration, we show that this configuration of circuit-hyperplanes of the tic-tac-toe matroid is also a configuration that prevents linear representability of a sparse paving matroid.
We present a new, clear and complete description
of the sparse paving matroids on nine elements
that contain this configuration.
In particular, we identify two maximal sparse
paving matroids whose relaxations contain all of them.
Finally, we prove that the dual tic-tac-toe matroids
do not satisfy the Ahslwede-K\"orner property, and
hence they are not almost entropic.

%
%
%
%
%
%
%

\section{Representations of Matroids and Polymatroids}
\label{sec: prelims}

Basic facts on matroids and polymatroids and some 
different ways in which they can be represented
are discussed in this section.

We begin by introducing some notation.
The number of elements of a finite set $X$ is denoted by $|X|$ 
and $\Par(X)$ denotes its power set.
Most of the times we use a compact notation for set unions
and we avoid the curly brackets for singletons.
That is, we write $XY$ for $X \cup Y$ and $Xy$ for $X \cup \{y\}$.
In addition, we write $X \sss Y$ for the set difference
and $X \sss x$ for $X \sss \{x\}$.
For a \emph{set function} 
$f \colon \Par(\grse) \to \Real$ on a finite set $\grse$
and sets $X,Y,Z \subseteq \grse$, we denote
\[
f(X \lon Y|Z)  = f(XZ) + f(YZ) - f(XYZ) - f(Z)
\]
and, in particular, 
$f(X \lon Y)  = f(X \lon Y| \emptyset)  = f(X) + f(Y) - f(XY)$
and $f(X|Z)  = f(X \lon X | Z) = f(XZ) - f(Z)$.
A set function on $\grse$
is \emph{monotone} if  
$f(X) \le f(Y)$ whenever $X \subseteq Y \subseteq \grse$
and it is \emph{submodular} if
$f(X) + f(Y) - (X \cap Y) - f(X \cup Y) \ge 0$
for every $X,Y \subseteq \grse$.

\begin{definition}\label{df:polym1}
A \emph{polymatroid} is a pair $(\grse,f)$
formed by a finite set $\grse$ (the \emph{ground set}) 
and a  monotone and submodular
set function $f$ on $\grse$ with $f(\emptyset) = 0$
(the \emph{rank function}). 
In the particular case that $f$ is integer-valued
and $f(x) \le 1$ for every $x \in \grse$, 
the polymatroid $(\grse,f)$ is a \emph{matroid}. 
\end{definition}


A polymatroid $(\grse,f)$ is 
\emph{linearly representable over a field $\kos$}
or \emph{$\kos$-linearly representable}, or 
simply \emph{$\kos$-linear},
if there exists a collection $(V_x)_{x\in \grse}$
of vector subspaces of a $\kos$-vector space $V$ such that
\[
f(X) = \dim \sum_{x \in X} V_x
\]
for every $X \subseteq \grse$.
That collection of subspaces is a
\emph{$\kos$-linear representation} of the polymatroid.
If $(\grse,f)$ is a matroid, then
$\dim V_x \le 1$ for every $x \in \grse$, and we can replace 
each subspace $V_x$ by a vector $v_x$ that spans the subspace.
That is, linear matroids are represented by
collections of vectors.
For a positive integer $\ell$, 
a matroid $(\grse,f)$ is 
\emph{$\ell$-folded $\kos$-linear} if
the polymatroid $(\grse, \ell f)$
is $\kos$-linear.
Every $\kos$-linear representation
$(V_x)_{x \in \grse}$ of the polymatroid
$(\grse, \ell f)$ is an 
\emph{$\ell$-folded $\kos$-linear 
representation} of the matroid $(\grse,f)$.

Consider a field extension $\kosl/\kos$ and 
a collection $(e_x)_{x \in \grse}$ 
of elements in $\kosl$.
For every $X \subseteq \grse$, 
take $\kos(X) = \kos((e_x)_{x \in X})$ and let
$f(X)$ be the transcendence degree of
the field extension  $\kos(X)/\kos$.
Then $(\grse,f)$ is a matroid.
In that situation, 
$(\grse,f)$ is \emph{algebraically representable over $\kos$}
or \emph{$\kos$-algebraically representable},
or simply \emph{$\kos$-algebraic},
and $(e_x)_{x \in \grse}$ 
is a \emph{$\kos$-algebraic representation} of $(\grse,f)$.
Every $\kos$-linear matroid is $\kos$-algebraic.

The joint Shannon entropies of
a collection of random variables determine
the rank function of an \emph{entropic} 
polymatroid~\cite{Fuj78,Fuj78-2}.
Limits of entropic polymatroids are called
\emph{almost entropic}.
It is well known that linearly representable
polymatroids are almost entropic, 
see~\cite{DFZ09} for a concise proof.
As a consequence, folded linear matroids are almost entropic.
Mat\'u\v{s}~\cite{Matus24} proved that 
algebraic matroids are almost entropic too.
Nevertheless, folded linear matroids 
are not necessarily algebraic~\cite{Ben16}.
For a graphical summary of the connections
between the classes of matroids determined by their
representations, see~\cite[Figure 1]{BBFP20}.

For a polymatroid $(\grse,f)$ and disjoint sets 
$Z_1, Z_2  \subseteq \grse$, the 
\emph{deletion} of $Z_1$ and the 
\emph{contraction} of $Z_2$ results in the polymatroid
on $\grse \sss Z_1 Z_2$ with rank function
$(f \backslash Z_1 |Z_2) (X) = f(X | Z_2)$.
Such polymatroids are the \emph{minors} of $(\grse,f)$.
In particular, we denote 
$(f \backslash Z) = (f \backslash Z_1 | \emptyset)$
and $(f |Z) = (f \backslash \emptyset |Z)$.  
Observe that minors of matroids are matroids.
The \emph{dual} of a matroid $\Mat = (\grse,f)$ 
is the matroid $\Mat^* = (\grse,f^*)$ with
\[
f^*(X) = |X| - f(\grse) + f(\grse \sss X)
\]
for every $X \subseteq \grse$.
Each of the classes of matroids or polymatroids that are
determined by the described representations is
minor-closed. 
The class of linearly representable polymatroids
and the class of folded linear matroids
are duality-closed.
It is unknown whether this applies to algebraic matroids or not,
while the class of almost entropic polymatroids is not
duality-closed~\cite{Kac18,Csi19}.

We introduce next some additional
terminology and basic facts about matroids.
The \emph{independent sets} of 
a matroid  $M = (\grse,f)$ 
are those with $f(X) = |X|$.
Every subset of an independent set is independent. 
The \emph{bases} of $M$ are the maximal independent sets,
while the minimal dependent sets are called \emph{circuits}.
All bases have the same number of elements,
which equals $f(\grse)$, the \emph{rank of the matroid}.
The \emph{closure} of $X \subseteq \grse$ 
is formed by all elements $x \in \grse$ such that
$f(Xx) = f(X)$.
A \emph{flat} is a set that
is equal to its closure. 
Flats of rank 2, 3, or $f(\grse)-1$ are called 
\emph{lines}, \emph{planes}, or \emph{hyperplanes}, respectively.

A matroid of rank $k$ is \emph{paving}
if the rank of every circuit is either $k$ or $k - 1$.
It is \emph{sparse paving} if,
in addition, all circuits of rank $k - 1$ are flats,
which in that situation are called \emph{circuit-hyperplanes}.
Therefore, every sparse paving matroid $M$ is
determined by the family $\cihy(\Mat)$ of its circuit-hyperplanes.
Observe that every set in $\cihy(\Mat)$ has exactly $k$ elements
and the intersection of any two different sets in $\cihy(M)$
has at most $k-2$ elements.
Moreover, every family of sets with those
properties determines a sparse paving matroid of rank $k$. 
If $\Mat$ and $\Mat'$ are sparse paving matroids
of the same rank on the same ground set and 
$\cihy(\Mat') \subseteq \cihy(\Mat)$, then 
$\Mat'$ is a \emph{relaxation} of $\Mat$.

The vertices of the \emph{Johnson graph} $J(n,k)$
are the subsets of $k$ elements out of a set of $n$ elements,
and each of its edges joins two
subsets with $k-1$ elements in the intersection.
Therefore, the sparse paving matroids of rank $k$
on $n$ elements are in one-to-one 
correspondence with the stable sets
of the graph $J(n,k)$.
A sparse paving matroid is \emph{maximal}
if it corresponds to a maximal stable set. 


We conclude the section with a discussion 
about single-element extensions.
See~\cite[Section~7.2]{Oxley11} for proofs 
and a more detailed exposition.
Consider sets $\grse,Z$ with $\grse \cap Z = \emptyset$.
The polymatroid $(\grse Z, g)$ is an 
\emph{extension} of the polymatroid $(\grse,f)$ 
if $f = (g \backslash Z)$.
If $(\grse,f)$ and $(\grse Z,g)$ are matroids and
$Z = \{z\}$, then  $(\grse z,g)$ is a
\emph{single-element extension} of $(\grse,f)$. 
A pair $(F_1,F_2)$ of flats of a matroid $(\grse,f)$ 
is \emph{modular} if $f(F_1 \lon F_2 | F_1 \cap F_2) = 0$.

\begin{definition}\label{def:modular cuts}
A \emph{modular cut} in a matroid $M$
is a family $\mathcal{F}$ of flats satisfying the following properties.
\begin{enumerate}
\item 
If $F \in\mathcal{F}$, then every flat containing $F$ 
is in $\mathcal{F}$ too.
\item If $F_1,F_2\in\mathcal{F}$ form a modular pair of flats, then
$F_1\cap F_2\in\mathcal{F}$.
\end{enumerate}
\end{definition}

For every single-element extension $(\grse z,g)$ of
a matroid $\Mat = (\grse,f)$, the flats $F$ of $\Mat$
with $g(Fz) = g(F)$ form a modular cut.
Conversely, for every modular cut $\mathcal{F}$ in $\Mat$,
there is a single-element extension $(\grse z,g)$ of $\Mat$ such that
a flat $F$ is in $\mathcal{F}$ 
if and only if $g(Fz) = g(F)$.

\section{Extension Properties of Matroids and Polymatroids}
\label{sec: prelims2}

We present in this section a unified approach
to several extension properties of 
matroids and polymatroids that can be found in the literature.
In contrast to previous works as~\cite{BaWa89},
our definitions involve iterated extensions.
The reader will find a summary about those properties
at the end of this section.

Given a polymatroid $(\grse,f)$ and sets $X,Y \subseteq \grse$,
a \emph{common information} for the pair $(X,Y)$ in $(\grse,f)$
is a set $Z \subseteq \grse$ such that
$f(Z | X)  = f(Z | Y) = 0$ and 
$f(X \lon Y | Z) = 0$.
A motivation for this concept and its name
is found in~\cite{DFZ09,FKMP20}.
Not every pair of sets admits a common information, 
but linear polymatroids can be extended to get one.
Indeed, if $(V_x)_{x \in \grse}$ is a $\kos$-linear representation
of $(\grse,f)$, take $z \notin \grse$, the subspace
\[
V_z = 
\left(
\sum_{x \in X} V_x
\right)
\bigcap
\left(
\sum_{y \in Y} V_y
\right)
\]
and the polymatroid $(\grse z, g)$ that is 
$\kos$-linearly represented by the collection
$(V_x)_{x \in \grse z}$.
Then $(\grse z, g)$ is a $\kos$-linear
extension of $(\grse,f)$ and
$z$ is a common information 
for $(X,Y)$ in $(\grse z,g)$.

\begin{definition}
For a polymatroid $(\grse,f)$ and
sets $X,Y \subseteq \grse$, 
a \emph{common information extension}, or
\emph{{\cip} extension} for short,
for the pair $(X,Y)$ is an extension
$(\grse Z,g)$ of the polymatroid $(\grse,f)$ such that
$Z$ is a common information
for $(X,Y)$ in $(\grse Z,g)$.
\end{definition}

\begin{definition}
\label{def:kci}
We recursively define \emph{$k$-{\cip} polymatroids}.
Every polymatroid is $0$-{\cip}.
For a positive integer $k$,  
a polymatroid is $k$-{\cip} if,
for every pair of subsets of the ground set, 
it admits a {\cip} extension that is a 
$(k-1)$-{\cip} polymatroid.
A polymatroid satisfies the
\emph{common information property},
or it is a \emph{{\cip} polymatroid},
if it is $k$-{\cip} for every 
positive integer $k$.
\end{definition}

\begin{definition}
\label{def:classci}
A class of polymatroids satisfies
the \emph{common information property}
if each of its members does and the {\cip} extensions
can be performed inside the class. 
\end{definition}

Observe that a class of polymatroids satisfies
the common information property
if and only if, for every polymatroid in the class
and every pair of subsets of the ground set, 
there exists a {\cip} extension inside the class.
The next result follows from the discussion opening this section.

\begin{proposition}
\label{st:civs}
For every field $\kos$,
the class of $\kos$-linearly representable polymatroids 
satisfies the common information property.
\end{proposition}


The previous description of the common information property
provides a template to introduce other
extension properties. 
In particular, 
Definitions~\ref{def:kci} and~\ref{def:classci}
are easily adapted and we are not going to repeat them.
We continue with an extension property for matroids, the
\emph{generalized Euclidean property},
which is based on the 
intersection property with the same name 
discussed in~\cite{BaWa89}.

\begin{definition}
For a matroid $(\grse,f)$ and a non-modular pair of flats $(F_1,F_2)$,
a \emph{generalized Euclidean extension},
or  \emph{{\gep} extension}, for the pair $(F_1,F_2)$ 
is a  matroid $(\grse z,g)$
that is a single-element extension of $(\grse,f)$
satisfying $g(z|F_1) = g(z|F_2) = 0$ and 
$g(F_1 z \cap F_2 z) = f(F_1 \cap F_2) + 1$.
That is, the extension increases 
the rank of the intersection of those flats.
\end{definition}

And similarly to Proposition~\ref{st:civs}, we have the following result. See~\cite{BaWa89} for more details.
\begin{proposition}
For every field $\kos$, the class of
$\kos$-linearly representable matroids satisfies the
generalized Euclidean property.
\end{proposition}


Different extensions can determine the
same extension property.
This is the situation for
{\gep} extensions
and complete Euclidean extensions, 
which are defined next. 
 
\begin{definition}
For a matroid $(\grse,f)$ and a pair of flats $(F_1,F_2)$,
a \emph{complete Euclidean extension},
or  \emph{{\cep} extension}, for the pair $(F_1,F_2)$ 
is a matroid $(\grse Z,g)$, which is an extension of $(\grse,f)$,
such that $g(Z|F_1) = g(Z|F_2) = 0$
and $(F_1 Z, F_2 Z)$ is 
a modular pair of flats in $(\grse Z,g)$.
Observe that one can take $Z = \emptyset$
if $(F_1,F_2)$ is a modular pair of flats.
\end{definition}

Observe that a {\gep} extension
can be trivially obtained from a {\cep} extension
and, conversely, every {\cep} extension
is the result of a sequence of {\gep} extensions.
Therefore, they determine the same extension property.
Nevertheless, 
the feasibility of computer-aided explorations
may depend on the chosen extension.

A matroid $\Mat = (\grse,f)$ of rank $d$
satisfies the \emph{Levi's intersection property}
if for every $d-1$ hyperplanes there is an extension of
$\Mat$ in which they meet in at least one point. 
That property,
which is discussed in~\cite{BaWa89},
can be used as well to describe
an extension property, 
but it is again equivalent to the generalized Euclidean property. 
That fact is well known,
but we could not find any
detailed proof in the literature.

\begin{proposition}
The extension property for matroids defined
from the Levi's intersection property
is equivalent to the 
generalized Euclidean property.
\end{proposition}

\begin{proof}
Let $\Mat = (\grse,f)$ be a matroid of rank $d$.
A point in the intersection of $d-1$ hyperplanes
$(H_1, \ldots, H_{d-1})$ of $\Mat$
can be found by a series of {\cep} extensions,
beginning with the pair $(H_1,H_2)$, and
then the pair formed by the intersection
of those hyperplanes and $H_3$, and so on. 
For the converse, consider a non-modular pair of flats $(F_1,F_2)$ and take 
$r_1 = f(F_1)$, $r_2 = f(F_2)$,
$s = f(F_1 F_2)$, and $t = f(F_1 \cap F_2)$.
Since the pair is non-modular, 
$t < r_1 + r_2 - s$.
For $i = 1,2$, take a basis $B_i$ of $F_i$
and a set $B'_i$ such that
$B_i B'_i$ is a basis of the flat 
spanned by $F_1 F_2$.
In addition, take a set $B''$ such that both
$B_1 B'_1 B''$ and $B_2 B'_2 B''$
are bases of $\Mat$.
Observe that $|B'_i| = s - r_i$ 
and $|B''| = d - s$.
For every $x \in B'_1 B''$,
take the hyperplane $H_x$ spanned by
$B_1 B'_1 B'' \sss x$ while, 
for every $x \in B'_2$,
the hyperplane $H_x$ will be the one
spanned by $B_2 B'_2 B'' \sss x$.
Take $C = B'_1 B'_2 B''$.
Then $(H_x)_{x \in C}$ is a collection of
$d + s - r_1 - r_2$
hyperplanes such that each of them 
contains $F_1$ or $F_2$ and
their intersection coincides with
$F_1 \cap F_2$. 
Take a basis $B$ of $M$ that contains a basis 
$D$ of $F_1 \cap F_2$.
For every $x \in D$,
take the hyperplane $H_x$ 
spanned by $B \sss x$.
Since $|D| = t$, 
we have collected up to now
$d + s - r_1 - r_2 + t$ hyperplanes,
a quantity that is not larger than $d-1$.
Each element in their intersection
is in $F_1 \cap F_2$ but it is not in the closure of 
$D$, which implies that there is no such element.
Finally, the Levi's intersection property
implies that $\Mat$ admits a 
single element extension $(\grse z,g)$
such that $z$ is in the intersection
of the hyperplanes $(H_x)_{x \in C D}$,
but not in the span of $F_1 \cap F_2$.
Therefore,  $(\grse z,g)$ is a 
$\gep$ extension for the pair $(F_1,F_2)$.
\end{proof}

The extension property
for the class of algebraic matroids 
that is defined next
is a direct consequence of the 
Ingleton-Main lemma~\cite{InMa75}. The
\emph{Dress-Lov\'asz property} is a more restrictive
extension property for algebraic matroids that
is derived from the generalization of
that lemma presented in~\cite{DrLo87}.
The reader is referred to~\cite{Bollen18}
for a description of the Dress-Lov\'asz extension property.

\begin{definition}
For a matroid $(\grse,f)$ and three lines
$\ell_1, \ell_2, \ell_3$ with 
$f(\ell_i \ell_j) = 3$ if $1 \le i < j \le 3$
and $f(\ell_1 \ell_2 \ell_3) = 4$,
an \emph{Ingleton-Main extension}, 
or \emph{{\imp} extension}, for the triple 
$(\ell_1, \ell_2,\ell_3)$
is a single-element extension
$(\grse z,g)$ of $(\grse,f)$ such that $g(z) = 1$ and
$g(z | \ell_i) = 0$ for $i = 1,2,3$.
That is, the three lines intersect in one point
in the extension.
\end{definition}

\begin{proposition}
For every field $\kos$, 
the class of $\kos$-algebraic matroids
satisfies the Ingleton-Main property.
\end{proposition}


We present next two extension properties
for almost entropic polymatroids.
The first one is based on the
Ahlswede-K\"orner lemma~\cite{AhKo77,AhKo06}
as stated in~\cite[Lemma~2]{Kac13}.
The second one is based on the copy lemma, 
which was proved in~\cite{DFZ11}
and was implicitly used before in~\cite{ZhYe98}
to find the first known non-Shannon information inequality.
%

\begin{definition}
For a polymatroid $(\grse,f)$ and sets 
$X, Y \subseteq \grse$,
an extension $(\grse Z, g)$ of $(\grse,f)$ is 
an \emph{Ahlswede-K\"orner extension},
or \emph{{\akp} extension}, for the pair 
$(X, Y)$ if the following conditions are satisfied.
\begin{itemize}
\item  
$g(Z|X) = 0$
\item
$g(X' | Z)=g(X' |Y)$ 
for every $X' \subseteq X$
\end{itemize}
\end{definition}

\begin{definition}
For a polymatroid $(\grse,f)$ and sets 
$X_1,X_2,Y \subseteq \grse$,
an extension $(\grse Z, g)$ of $(\grse,f)$ is 
a \emph{copy lemma extension},
or \emph{{\clp} extension}, for $(X_1,X_2,Y)$ 
if the following conditions are satisfied.
\begin{itemize}
\item  
There is a bijection 
$\varphi \colon Y \to Z$ that determines an isomorphism
between the polymatroids
$(X_1 Y,(g \backslash A))$ and 
$(X_1 Z,(g \backslash B))$,
where $A = \grse Z \sss X_1 Y$ and $B = \grse Z \sss X_1 Z$ 
\item
$g(Z \lon X_2 Y | X_1) = 0$
\end{itemize}
\end{definition}

\begin{proposition}
The class of almost entropic polymatroids
satisfies both the Ahlswede-K\"orner and copy lemma properties.
\end{proposition}

If a matroid satisfies the generalized Euclidean property,
then it satisfies the Ingleton-Main property too,
because the latter deals with intersection of lines,
that is, flats of rank $2$.
Every $\cip$-polymatroid 
is an $\akp$-polymatroid~\cite{FKMP20}. 
An additional result 
on the connection between the common information
and Ahlswede-K\"orner properties is provided by the next proposition,
whose elementary proof is given 
in~\cite[Proposition~III.16]{FKMP20}.

\begin{proposition}
\label{st:civsak}
Let $(\grse Z,g)$ be an {\akp} extension of $(\grse,f)$ for the pair $(X,Y)$.
Then $g(Z) = g(X \lon Y)$.
Consequently, $(\grse Z,g)$ is a {\cip} extension 
for the pair $(X,Y)$ if and only if $g(Z|Y) = 0$.
\end{proposition}

The common information and
complete Euclidean properties are very similar.
The main difference is that
the latter is an extension property of matroids, 
that is, we require that the extension is a matroid, 
while the former is an extension property of polymatroids. 
Folded linear matroids satisfy the common information property
but it is not clear that they satisfy the complete Euclidean property,
because it could be that the extension derived 
from a given folded linear representation is not a matroid.
Nevertheless, no example of a matroid
satisfying the common information property but not
the complete Euclidean property is known.

Proposition~\ref{st:IMvsAK} describes a connection
between the Ingleton-Main and Ahlswede-K\"orner extensions.
Nevertheless, we have to remember that the first one applies
to matroids and the second one to polymatroids.
A proof for the following technical result
can be found in~\cite[Proposition~III.13]{FKMP20}.

\begin{lemma}
\label{st:ciprop1}
Consider a polymatroid  $(\grse,f)$ and  
subsets $X,Y, U \subseteq \grse$
such that $f(U|X) = f(U|Y) = 0$.
Then $f(U) \le f(X \lon Y)$.
Moreover, if $Z \subseteq \grse$ 
is a common information for $(X,Y)$,
then $f(U|Z) = 0$.
\end{lemma}

\begin{proposition}
\label{st:IMvsAK}
Consider a polymatroid $(\grse,f)$
and sets $L_0, L_1, L_2 \subseteq \grse$ with
$f(L_i) = 2$, $f(L_i L_j) = 3$ if $i \ne j$, and
$f(L_0 L_1 L_2) = 4$.
Let $(\grse Z,g)$ be an {\akp} extension for the pair
$(L_1 L_2,L_0)$.
Then $g(Z|L_i) = 0$ for each $i = 0,1,2$.
\end{proposition}

\begin{proof}
By Proposition~\ref{st:civsak},
$g(Z) = g(L_1 L_2 \lon L_0) = 1$
and, by the definition of {\akp} extension,
$g(L_1|Z) = g(L_1|L_0) = 1$.
Therefore, $g(L_1 Z) = g(Z) + g(L_1|Z) = 2$, 
and hence $g(Z|L_1) = 0$.
By symmetry, $g(Z|L_2) = 0$.
Since $L_0$ is a common information for the 
pair $(L_0 L_1,L_0 L_2)$ and 
$g(Z|L_0 L_1) = g(Z|L_0 L_2) = 0$,
by Lemma~\ref{st:ciprop1} $g(Z|L_0) = 0$.
\end{proof}


The generalized Euclidean property
is preserved by taking minors~\cite{BaWa89},
and the same applies to the Ingleton-Main 
and Dress-Lov\'asz properties~\cite{Bollen18}.
We prove in Propositions~\ref{st:cimin} 
and~\ref{st:akmin} that both the
common information and Ahlswede-K\"orner properties
are preserved by taking minors as well.
Since it is obvious that this is the case for deletions, 
we consider only contractions.

\begin{proposition}\label{st:cimin}
Consider a polymatroid $(\grse,f)$ and sets 
$U \subseteq \grse$ and $X,Y \subseteq \grse\sss U$. 
Let $(\grse Z,g)$ be a {\cip} extension of $(\grse,f)$
for the pair $(XU,YU)$. 
Then $(\grse Z \sss U, (g|U))$ 
is a {\cip} extension of 
$(\grse\sss U, (f|U))$ for the pair $(X,Y)$.
\end{proposition}

\begin{proof}
Observe that the minor $(\grse Z \sss U, (g|U))$
of  $(\grse Z, g)$ is an extension of 
$(\grse \sss U, (f|U))$.
Since $g(U|Z) = 0$ 
by Lemma~\ref{st:ciprop1}, 
one can check with
a straightforward calculation that
$(g|U)(X \lon Y | Z) = g(XU \lon YU | Z) = 0$. 
Finally, 
$(g|U)(Z|X) = g(Z| XU) = 0$
and, analogously, $(g|U)(Z|Y) = 0$.
\end{proof}

\begin{proposition}\label{st:akmin}
Consider a polymatroid $(\grse,f)$ and sets 
$U \subseteq \grse$ and 
$X, Y \subseteq \grse\sss U$.
Let $(\grse Z, g)$ be an {\akp} extension of $(\grse,f)$
for the pair $(XU,YU)$.
Then $(\grse Z \sss U, (g|U))$ 
is an {\akp} extension of 
$(\grse\sss U, (f|U))$ for the pair $(X,Y)$.
\end{proposition}

\begin{proof}
On one hand,
$(g|U)(Z |X) = g(Z|X U) = 0$.
On the other hand, 
$g(U|Z) = g(U | YU) = 0$, and hence
\[
(g|U)(X'| Z) = g(X' | ZU) = g(X'|Z) = 
g(X'| YU) = (g|U) (X'|Y)
\]
for every $X' \subseteq X$.
\end{proof}


The extension properties we discussed in this section
are not useful for matroids of rank~$3$, 
because all of them satisfy the
common information property
and the Ahlswede-K\"orner property.
A proof for this well-known fact 
can be found in~\cite[Proposition~3.18]{BBFP20}.
Nevertheless, there are matroids of rank $3$ that are not
almost entropic~\cite{Mat99} as, for example, the non-Desargues matroid.

Summarizing, we have presented in this section
six estension properties.
Three of them apply to matroids,
namely $\gep$, $\cep$ and $\imp$.
The first two, which are equivalent, 
are satisfied by the class of folded linear matroids,
while the third one is satisfied by all algebraic matroids. 
Every $\gep$ matroid is $\imp$ too.
The other three, which are $\cip$, $\akp$, and $\clp$,
apply to polymatroids.
Only the first two are extensively discussed in this paper.
The $\cip$ property is closely related $\cep$
and it is satisfied by all linear polymatroids.
Every $\cip$ polymatroid is $\akp$ as well.
All almost entropic polymatroids
satisfy the $\akp$ and $\clp$ properties.

\section{Matroids on Eight Elements}
\label{sec:8pointsRevisited}

In this section, we discuss the
extension properties  of the matroids
of rank $4$ on $8$ elements
that are not linearly representable.
All of them are sparse paving matroids.
We review the known facts and
we present a new result about
the relaxations of the matroid $P_8$
that are not folded linear.
Namely, we prove that they
are not $\cip$ matroids.
The proof is based on computer explorations, 
but we prove in a human readable way that
two of them do not satisfy the
generalized Euclidean property.

We identify the ground set 
$\grse = \{0,1, \ldots, 7\}$ of those matroids
with the set of vertices $(x,y,z) \in \{0,1\}^3$
of a $3$-dimensional cube.
Specifically, each element in $\grse$ is identified
with the vertex corresponding to its binary representation.
For instance,  $2$ is identified to $(0,1,0)$
and $6$ to $(1,1,0)$.
The vertices of that cube can be identified with
the eight points in the affine space 
of dimension~$3$ over the binary field $\field_2$.
The circuit-hyperplanes of the sparse paving matroid $AG(3,2)$
are the  $14$ planes in that space.
Clearly, $AG(3,2)$ is a maximal sparse
paving matroid.
It is $\kos$-linear if and only if
$\kos$ has characteristic $2$~\cite{Oxley11}.

We begin with a very simple and
well known application of extension properties,
namely proving that the V\'amos matroid
is not almost entropic.
The V\'amos matroid $\vamos = (\grse,f)$ 
is the relaxation of $AG(3,2)$
whose five circuit-hyperplanes are
\[
0123,0145,2367,4567,2345
\]
Take $L_0 = 01$, $L_1 = 23$, $L_2 = 45$, and $L_3 = 67$
and consider the polymatroid $\vamosp = (\grse_o,h)$ with
ground set $\grse_o = \{L_0,L_1,L_2,L_3\}$
and rank function $h$ such that the rank of 
each singleton equals $2$, the sets with two elements
have rank $3$ except for $h(L_0 L_3) = 4$,
while the rank of all larger sets is equal to $4$.
That is, the rank function of $\vamosp$ is the one
induced by the one of $\vamos$.

\begin{lemma}
\label{st:vamosak}
The polymatroid $\vamosp$ is not $1$-{\akp},
and hence it does not satisfy the Ahlswede-K\"orner property. 
\end{lemma}

\begin{proof}
Assume that there exists an {\akp} extension
$(\grse_o Z,g)$ for the pair $(L_1 L_2,L_0)$.
Then $g(Z|L_i) = 0$ for each $i = 0,1,2$
by Proposition~\ref{st:IMvsAK}.
Moreover, $g(Z|L_3) = 0$ by Lemma~\ref{st:ciprop1}
because $L_3$ is a common information for the 
pair $(L_3 L_1,L_3 L_2)$ and
$g(Z|L_3 L_1) = g(Z|L_3 L_2) = 0$.
Finally, applying Lemma~\ref{st:ciprop1} again,
$1 = g(Z) \le g(L_0 \lon L_3) = 0$, a contradiction. 
\end{proof}

As a consequence of the previous proof,
there is no {\akp} extension of the V\'amos matroid 
for the pair $(L_1 L_2,L_0)$.
Therefore, the V\'amos matroid is not almost entropic, and
hence is not algebraic.
The same arguments apply to every matroid
presenting the configuration
of the circuit-hyperplanes of the V\'amos matroid. 
Specifically, a matroid $(\grse,f)$ contains the 
\emph{V\'amos configuration} if there are
four lines $(L_i)_{0 \le i \le 3}$
such that the polymatroid induced
by $(\grse,f)$ on $\grse_o = \{L_0,L_1,L_2,L_3\}$
coincides with $\vamosp$.
Up to isomorphism, there are exactly $39$ matroids 
on eight elements containing the V\'amos configuration,
which are precisely the matroids on eight elements
that do not satisfy the Ingleton inequality~\cite{MaRo08}. 
All of them are sparse paving and relaxations of $AG(3,2)$.

Another relaxation of $AG(3,2)$ is the sparse-paving matroid
$L_8$, whose circuit-hyperplanes are the six faces of the cube and
the two twisted planes $0356$, $1247$ of $AG(3,2)$.
It is $\kos$-linear if and only if $|\kos| \ge 5$~\cite{Oxley11}.
The sparse paving matroid $L'_8$ that is obtained
from $L_8$ by relaxing one of the twisted planes 
is not linear, but it is folded linear~\cite{BBFP20}.
Moreover, it is algebraic over all fields with positive 
characteristic~\cite[Example~35]{BCD18}.

\begin{figure}
\centering
\begin{tikzpicture}[scale=1.5]

\filldraw[black]  (1,1) circle (2pt) ;
\node at (1.2,1.2) {\large 6};

\filldraw[black]  (-1,1) circle (2pt);
\node at (-1.2,1.2) {\large 2};

\filldraw[black]  (-1,-1) circle (2pt);
\node at (-1.2,-1.2) {\large 0};

\filldraw[black]  (1,-1) circle (2pt);
\node at (1.2,-1.2) {\large 4};

\filldraw[black]  (1.41,0) circle (2pt);
\node at (1.61,0) {\large 5};

\filldraw[black]  (0,1.41) circle (2pt);
\node at (0,1.61) {\large 7};

\filldraw[black]  (-1.41,0) circle (2pt);
\node at (-1.61,0) {\large 3};

\filldraw[black]  (0,-1.41) circle (2pt);
\node at (0,-1.61) {\large 1};

\draw[thick] (0,1.41) -- (1.41,0);
\draw[thick] (0,1.41) -- (-1.41,0);
\draw[thick] (0,-1.41) -- (1.41,0);
\draw[thick] (0,-1.41) -- (-1.41,0);

\draw[thick] (1,1) -- (1,0.41);
\draw[thick,dashed] (1,0.41) -- (1,-0.41);
\draw[thick] (1,-1) -- (1,-0.41);

\draw[thick] (-1,1) -- (-1,0.41);
\draw[thick,dashed] (-1,0.41) -- (-1,-0.41);
\draw[thick] (-1,-1) -- (-1,-0.41);

\draw[thick] (1,1) -- (0.41,1);
\draw[thick,dashed] (0.41,1) -- (-0.41,1);
\draw[thick] (-1,1) -- (-0.41,1);

\draw[thick] (1,-1) -- (0.41,-1);
\draw[thick,dashed] (0.41,-1) -- (-0.41,-1);
\draw[thick] (-1,-1) -- (-0.41,-1);

\draw[thick] (1,1) -- (0.41,1);
\draw[thick,dashed] (0.41,1) -- (-0.41,1);
\draw[thick] (-1,1) -- (-0.41,1);

%
%

\end{tikzpicture}
\caption{Geometric representation of $P_8$}
\label{fg:cubep8}
\end{figure}

The ten circuit-hyperplanes of the sparse
paving matroid $P_8$ are
\[
0246, 1357, 0217, 4617, 2635, 0435, 
0637, 0615, 2413, 2457
\] 
Following~\cite{GOVW00}, we consider the
geometric representation of 
$P_8$ that is obtained
by rotating half right angle 
the upper plane $1357$ of the cube, 
as in Figure~\ref{fg:cubep8}.
It is not difficult to check that $P_8$
is a maximal sparse paving matroid.
By relaxing $2635$ from $P_8$ we obtain
the sparse paving matroid $\pone$.
By relaxing $2635$ and $1357$, we obtain $\ptwo$ while
the relaxation of $2635$ and $0246$ produces the matroid $\ptwop$.
Finally, $\pthree$ is the result of 
relaxing those three circuit-hyperplanes.
Observe that $P_8$, $\pone$ and $\pthree$
are self-dual (but not identically self-dual)
while the dual of $\ptwo$ is isomorphic to $\ptwop$.
The matroid $P_8$ is $\kos$-linear
if and only if the characteristic of $\kos$ 
is different from $2$~\cite{Oxley11}. 
By a generalization of the method described 
in~\cite[Section~6.4]{Oxley11}, 
it was proved in~\cite{BBFP20}
that $\pthree$ is folded linear but not linear
while neither $\pone$, $\ptwo$, nor $\ptwop$ are folded linear.
Moreover, $\pthree$ is algebraic
over all fields with positive 
characteristic~\cite{Lindstrom1986}.
It is not known whether $\pone$, $\ptwo$ and $\ptwop$
are algebraic, almost entropic, or neither.

By computer-aided explorations, 
we checked that those three matroids are not $4$-\cip.  
The pairs $(X_i, Y_i)$ for which those common information extensions are
not possible are shown in Table~\ref{tab:4CI_combs}.

\begin{table}[ht]
\centering
\begin{tabular}{@{}ccccc@{}}
\toprule
Matroid & $(X_1,Y_1)$  & $(X_2,Y_2)$  & $(X_3,Y_3)$  & $(X_4,Y_4)$  \\ \midrule
$\pone$   & $01$, $56$ & $17$, $35$ & $67$, $03$ & $13$, $57$ \\
$\ptwo$  & $01$, $56$ & $17$, $35$ & $67$, $03$ & $13$, $57$ \\
$\ptwop$ & $01$, $27$ & $06$, $24$ & $67$, $14$ & $04$, $35$ \\ \bottomrule
\end{tabular}
\caption{Pairs of sets for which iterated $\cip$ extensions at depth 4 do not exist
for the matroids $\pone$, $\ptwo$, and $\ptwop$}
\label{tab:4CI_combs}
\end{table}

We conclude this section with 
a human readable proof
for the fact that the
matroids  $\pone$ and $\ptwo$ do not satisfy the
generalized Euclidean property. 
Let $(\grse,f)$ be one of the matroids $\pone$ or $\ptwo$ and
let $(\grse z_1 z_2 z_3, g)$ be a triple {\gep} extension
of  $(\grse,f)$, where $z_1, z_2, z_3$ correspond to
the non-modular pairs of lines 
$(16, 47),(12, 07),(06,24)$, respectively.
Since $z_1$ is on the line $47$ and $z_2$ is
on the line $07$, the line $z_1 z_2$ is on the
plane $047$, and hence the lines $04$ and $z_1 z_2$ are coplanar.
Analogously, the lines $26$ and $z_1 z_2$ are coplanar.
Observe that $z_1$ is in the intersection 
of the planes $0615$ and $2457$,
which is equal to the line $z_3 5$.
Moreover, the line $z_3 3$ is the intersection of the planes
$2413$ and $0673$, which contains $z_2$.
It follows that the lines $z_1 z_2$ and $35$ are coplanar.
Therefore, the lines  $26$, $04$, $z_1 z_2$ and $35$
form the V\'amos configuration, which implies that 
$(\grse z_1 z_2 z_3,g)$ is not $1$-{\gep}.
We can conclude that $(\grse,f)$ is not a {\gep} matroid.


\section{Matroids with the
Tic-Tac-Toe Configuration}\label{sec:TTTmats}

The tic-tac-toe matroid
was introduced as a possible counterexample
to prove that the class of algebraic matroids is
not closed by duality.
It is a sparse paving matroid of rank five
on nine elements with eight circuit-hyperplanes.
It satisfies the Ingleton inequality but it is not a 
{$\gep$} matroid~\cite{AlHo95}.
While it satisfies the Dress-Lov\'asz property~\cite{Bollen18},
it is not known whether it is algebraic or not.
Nevertheless, the dual of  the tic-tac-toe matroid is not
an {\imp} matroid~\cite{Hoc97}, and hence it is not algebraic.

We prove in the following that the tic-tac-toe matroid is not {\cip}
and its dual is not {\akp}, 
and hence the former is not folded linear and
the latter is not almost entropic. 
That applies as well to the sparse paving matroids 
of rank five on nine elements that contain
eight circuit-hyperplanes in the same configuration
as the tic-tac-toe matroid.
By a computer-aided exploration on 
the database~\cite{RoMaDatabase},
we found that there are exactly $181$ such matroids\confv{ (see~\cite{BFP23})}\fullv{, which are listed in the Appendix}. 
Once they were determined, we realized that their circuit-hyperplanes
can be described in terms of the points and lines on
the affine plane over $\field_3$.
Moreover, we identified two maximal sparse paving matroids 
whose relaxations contain all of them.


Let $\grse = \field_3 \times \field_3$ be the set of points 
$(x,y)$ on the affine plane over the field $\field_3$.
The 12 lines on that affine plane,
which are represented
in Figure~\ref{fg:apf33}, are partitioned into
four sets of three parallel lines.
Namely, for $i \in \field_3$,
\begin{itemize}
\item
the lines $A_i$ with equation $y = i$
\item
the lines $B_i$ with equation $x = i$,
\item
the lines $C_i$ with equation $x-y = i$,
\item
the lines $D_i$ with equation $x+y = i$.
\end{itemize}
We describe in the following several
sparse paving matroids of rank $5$ with ground set $\grse$
that are related to the 
tic-tac-toe matroid.
The $9$ circuit-hyperplanes
of the matroid $\tto$ are the sets
$A_i B_j$ with $(i,j) \in \field_3 \times \field_3$.
The \emph{tic-tac-toe matroid}, which is denoted here by $\ttt$, 
is obtained from $\tto$ 
by relaxing the circuit-hyperplane $A_0 B_0$.
Observe that the relaxation of
any other circuit-hyperplane instead
of $A_0 B_0$ produces an isomorphic matroid.

\begin{figure}
\centering
\begin{tikzpicture}[scale=1]

\filldraw[black]  (0,0) circle (2pt);
\filldraw[black]  (0,1) circle (2pt);
\filldraw[black]  (0,2) circle (2pt);

\filldraw[black]  (1,0) circle (2pt);
\filldraw[black]  (1,1) circle (2pt);
\filldraw[black]  (1,2) circle (2pt);

\filldraw[black]  (2,0) circle (2pt);
\filldraw[black]  (2,1) circle (2pt);
\filldraw[black]  (2,2) circle (2pt);

\draw (0,0) -- (0,2);
\draw (1,0) -- (1,2);
\draw (2,0) -- (2,2);

\draw (0,0) -- (2,0);
\draw (0,1) -- (2,1);
\draw (0,2) -- (2,2);

%
%
%

\draw (0,0) -- (2,2);
\draw (0,1) -- (1,2);
\draw (1,0) -- (2,1);

\draw (1,2) .. controls  (2,3) and (4,3) .. (2,0);
\draw (1,0) .. controls  (0,-1) and (-2,-1) .. (0,2);

\draw (0,2) -- (2,0);
\draw (0,1) -- (1,0);
\draw (1,2) -- (2,1);

\draw (2,1) .. controls  (3,0) and (3,-2) .. (0,0);
\draw (0,1) .. controls  (-1,2) and (-1,4) .. (2,2);

\end{tikzpicture}
\caption{The affine plane over $\field_3$}
\label{fg:apf33}
\end{figure}

Alfter and Hochst\"attler~\cite{AlHo95} proved that
the tic-tac-toe matroid $\ttt$ does not
satisfy the generalized Euclidean property.
We prove next that 
it does not satisfy the common information property.
As a consequence, $\ttt$ is not folded linear.

\begin{proposition}
\label{st:tttnoci}
The tic-tac-toe matroid $\ttt$ does not satisfy 
the common information property. 
\end{proposition}

\begin{proof}
Clearly, $Z_0 = \{(0,0)\}$ is a common information
for the pair $(A_0,B_0)$ in $\ttt$.
Let $(\grse Z,g)$ be a $\cip$ extension of $\ttt = (\grse,f)$
for the pair $(A_{-1}, A_1)$.
Since $B_i$ is a common information
for the pair $(B_i A_{-1},B_i A_1)$
and $g(Z|B_i A_{-1}) = g(Z|B_i A_{1}) = 0$
it follows by Lemma~\ref{st:ciprop1}
that $g(Z|B_i) = 0$ for each  $i \in \field_3$. 
In particular, $Z$ is a common information
for the pair $(B_{-1},B_1)$
and, by symmetry, $g(Z|A_0) = 0$.
Since $g(Z|A_0) = g(Z|B_0) = 0$, 
we infer that $g(Z|Z_0) = 0$
by Lemma~\ref{st:ciprop1}, and hence 
$g(Z_0|Z) = 0$ because $g(Z) = g(Z_0)$.
That is in contradiction to 
$g(A_1 Z) \ne g(A_1 Z_0)$.
\end{proof}


The dual $(\ttt)^*$ of the tic-tac-toe matroid is not algebraic
because it does not satisfy the 
Ingleton-Main property~\cite[Proposition 5]{Hoc97}.
By using Proposition~\ref{st:IMvsAK},
the proof of that result is easily adapted
to show that it does not satisfy the 
Ahlswede-K\"orner property. 
Therefore, $(\ttt)^*$ is not almost entropic. 

\begin{proposition}
\label{st:tttdnoak}
The dual $(\ttt)^*$ of the tic-tac-toe matroid
does not satisfy the Ahlswede-K\"orner property.
\end{proposition}

\begin{proof}
For each pair $(i,j)$ in $\field_3 \times \field_3$,
take $L^{i}_{j} = \{(i,k),(i,\ell)\}$
with $\{j,k,\ell \} = \field_3$.
Observe that $(\ttt)^*$ is a sparse paving matroid
of rank $4$.
Its circuit-hyperplanes are
the 8 sets $\grse \sss A_i B_j$ with 
$(i,j) \ne \{(0,0)\}$,
which coincide with
the sets of the form 
$L^i_j L^{i'}_j$ with $i \ne i'$
other than $L^{-1}_0 L^1_0$.
Let $(\grse Z_{-1} Z_1,g)$ be a double {\akp} extension of $(\ttt)^*$,
where $Z_{j}$ corresponds to the pair 
$(L^{-1}_{j} L^{1}_{j},L^{0}_{j})$.
Recall that, since we are dealing
with {\akp} extensions, 
$(\grse Z_{-1} Z_1,g)$ may not be a matroid.
By Propositions~\ref{st:civsak}
and~\ref{st:IMvsAK}, $g(Z_j) = 1$ and
$g(Z_j | L^{i}_{j}) = 0$ for each
$i \in \field_3$ and $j \in \field_3 \sss \{0\}$.
Take $L = \{Z_{-1}, Z_1\}$.
Since 
\(
g(Z_{-1} | L^{-1}_{-1}) =
g(Z_{1} | L^{1}_{1}) = 0
\) 
and 
\(
g(L^{1}_{1} | L^{-1}_{-1}) = g(L^{1}_{1})
\),
it is clear that $g(L) = 2$.
If $i \ne i'$, then
$g(L^{i}_0 | L^{i'}_1 ) = g(L^i_0)$,
and hence $L \ne L^{i}_0$ because
$g(Z_1 | L^{i'}_1) = 0$. 
In addition, both $L$ and $L^{i}_{0}$ are in the span of
$\{(i,-1),(i,0),(i,1)\}$, which implies that
$g(L L^{i}_{0}) = 3$ for each $i \in \field_3$.
Therefore, the polymatroid induced by
$(\grse Z_{-1} Z_1,g)$ on
\(
\{L^{-1}_{0},L^{0}_{0},L,L^{1}_{0}\}
\)
is isomorphic to $\vamosp$
and the proof is concluded  
by Lemma~\ref{st:vamosak}.
\end{proof}

The previous propositions can be applied 
as well to the sparse paving matroids
on nine elements containing 
circuit-hyperplanes in the same configuration
as $\ttt$ or $(\ttt)^*$.
That is, the TTT matroids that are 
described in Definition~\ref{def:tttkinds}
and their duals.
We skip the proofs of the next lemma and
other similar results in this section
because they are elementary 
but quite cumbersome.

\begin{lemma}
\label{st:cihyT31}
Let $M$ be a sparse paving matroid on $\grse$ 
with $\cihy(\tto) \subseteq \cihy(\Mat)$.
Then every circuit-hyperplane in
$\cihy(M) \sss \cihy(\tto)$ is of the form $C_i D_j$
for some $(i,j) \in \field_3 \times \field_3$.
\end{lemma}

Let $\ttc$ be the sparse paving matroid on $\grse$
whose 18 circuit-hyperplanes are 
the sets $A_i B_j$ and $C_i D_j$ with
$(i,j) \in \field_3 \times \field_3$.
By Lemma~\ref{st:cihyT31},
$\ttc$ is the only maximal sparse paving matroid 
that has $\tto$ as a relaxation.

\begin{lemma}
\label{st:cihyT3}
Let $M$ be a sparse paving matroid on $\grse$ 
such that $\cihy(\ttt) \subseteq \cihy(M)$.
Then there is a unique pair
$(i_o,j_o) \in \grse \sss A_0 B_0$
such that every circuit-hyperplane in
$\cihy(M) \sss \cihy(\ttt)$ is either in 
$\cihy(\ttc )$ or is equal to
$(A_0 B_0 \sss (0,0)) \cup (i_o,j_o)$.
\end{lemma}

Consider the sparse paving matroid $\ttw$
with the circuit-hyperplanes of $\ttt$ 
together with $(A_0 B_0 \sss (0,0)) \cup (1,1)$.
Taking another point $(i_0,j_o) \in \grse \sss A_0 B_0$
instead of $(1,1)$ yields a matroid 
isomorphic to $\ttw$.

\begin{definition}
\label{def:tttkinds}
Let $\Mat$ be a sparse paving matroid 
of rank $5$ on $\grse$ with 
$\cihy(\ttt) \subseteq \cihy(\Mat)$
and $A_o B_0 \notin \cihy(M)$.
If $\cihy(\Mat) \subseteq \cihy(\ttc)$, then $\Mat$ is
a \emph{TTT matroid of the first kind}, 
while  $\Mat$ is a 
\emph{TTT matroid of the second kind}
if  $(A_0 B_0 \sss (0,0)) \cup (1,1) \in \cihy(\Mat)$.
\end{definition}

As a consequence of Lemmas~\ref{st:cihyT31} 
and~\ref{st:cihyT3},
every sparse paving matroid
of rank five on nine elements that contains
eight circuit-hyperplanes in the same configuration
as the tic-tac-toe matroid is isomorphic to a TTT matroid,
either of the first kind or the second kind.

Neither $C_{-1} D_{-1}$ nor
$C_{1} D_{-1}$ can be circuit-hyperplanes of
a TTT matroid of the second kind.
Because of that,
there is only one maximal sparse paving matroid on $\grse$
that admits $\ttw$ as a relaxation, which we denote by $\twc$.
Its $16$ circuit-hyperplanes are
\begin{itemize}
\item
$A_i B_j$ with $(i,j) \ne (0,0)$,
\item
$(A_0 B_0 \sss (0,0)) \cup (1,1)$, and
\item
$C_i D_j$ with
$(i,j) \ne (-1,-1)$ and $(i,j) \ne (1,-1)$.
\end{itemize}
Every TTT matroid of the second kind is 
a relaxation of $\twc$
and every TTT matroid of the first kind
is a relaxation of $\ttc$.

By using the same arguments as in the proofs
of Propositions~\ref{st:tttnoci} and~\ref{st:tttdnoak},
no TTT matroid satisfies the common information property,
while the Ahlswede-K\"orner property is not satisfied by 
any dual of a TTT matroid. 

By applying the method described in~\cite[Section~6.4]{Oxley11},
one can check that 
the matroid $\ttc$ is linearly representable,
but only over fields of characteristic $3$.
A representation over $\field_3$ is given by the following matrix,
whose columns are indexed by
the elements in $\grse$ in the order
$(-1,-1), (-1,0), (-1,1), (0,-1)$,
$(0,0), (0,1), (1,-1), (1,0), (1,1)$.
\[
\left(
\begin{array}{rrrrrrrrr}
1 & 0 & 0 & 0 & 0 & 1 & 1 & 1 & 1 \\
0 & 1 & 0 & 0 & 0 & 1 & -1 & -1 & 1 \\
0 & 0 & 1 & 0 & 0 & 1 & -1 & 1 & 0 \\
0 & 0 & 0 & 1 & 0 & 1 & 1 & 0 & -1 \\
0 & 0 & 0 & 0 & 1 & 1 & 0 & -1 & -1
\end{array}
\right)
\]
The matroid $\tto$ is linear over fields of every characteristic.
A linear representation is given in~\cite[Section~5]{BBFP20}.

The TTT matroids of the second kind that contain
the circuit-hyperplanes 
\begin{equation}
\label{eq:VamTTT}
C_0 D_{-1}, C_0 D_0, C_0 D_1,  A_1 B_1,  
(A_0 B_0 \sss (0,0)) \cup (1,1)
\end{equation}
are not almost entropic.
Indeed, let $\Mat$ be such a matroid.
By removing $(1,1)$ from each of
the sets in~(\ref{eq:VamTTT}), we obtain 
five circuit-hyperplanes of the minor  of $\Mat$
that results from the contraction of $(1,1)$.
It is easy to check that they form a V\'amos configuration.
For every TTT matroid of the second kind in that situation,
the contraction of $(1,1)$ yields a sparse paving matroid 
of rank $4$ on $8$ elements whose
$9$ circuit-hyperplanes are
obtained by removing $(1,1)$ from each of the sets
\[
C_0 D_{-1}, C_0 D_0, C_0 D_1,  A_1 B_1,  
(A_0 B_0 \sss (0,0)) \cup (1,1),
A_1 B_{-1}, A_1 B_0, A_0 B_1, A_{-1} B_1
\]
and hence that minor is isomorphic
to the matroid with tag 1509 in
the database by Royle and Mayhew~\cite{RoMaDatabase}.
There are exactly $10$ TTT matroids 
of the second kind in that situation,
which are specified in\fullv{the Appendix}\confv{~\cite{BFP23}}.
The matroid $\twc$ is one of them.

The representability of matroids on nine elements
by Frobenius-flocks over 
fields of characteristic $2$ 
was exhaustively explored by Bollen~\cite{Bollen18}.
By matching the results of that exploration
with our list of TTT matroids, 
we found that at least 62 of them are not
Frobenius-flock representable in characteristic $2$, and 
hence they are not algebraic over fields of characteristic $2$.
Those TTT matroids have at least $12$ circuit-hyperplanes.
They are listed in the \fullv{Appendix}\confv{extended version of this work~\cite{BFP23}}.
In particular, for all fields of positive characteristic,
to determine whether the tic-tac-toe matroid $\ttt$ is algebraic or not
remains an open problem.

\section*{Acknowledgment}
\confv{We thank the anonymous reviewer for carefully revising
our paper and providing 
valuable suggestions for its improvement.}
\fullv{We are grateful to Guus P. Bollen for his assistance with the matroid database and for sharing his expertise. We also thank anonymous reviewers for carefully revising our paper and providing valuable suggestions for its improvement.}



\fullv{

\appendix

\section{Other Results of Computer-Aided Explorations}

The experimental results presented in this appendix
were obtained by using the 
Gurobi\textsuperscript{TM} optimizer 
for solving linear programming problems, 
and the SageMath matroid package for the 
recursive implementation of the intersection 
properties and other matroid operations. 
The programs we used are available at \cite{Bam22}.

\subsection{Identifying TTT Matroids in the Existing Databases}\label{s:listTTT}


We identified the TTT matroids 
introduced in Section~\ref{sec:TTTmats} in the databases
by Royle and Mayhew~\cite{RoMaDatabase}
and by Bollen~\cite{Bollen18-1}. 
In particular, we found out that,
up to isomorphism, there are 181 such matroids,
which are listed in Table~\ref{tab:sparse pav TTT mats}.
They are identified by the tags  
from~\cite{RoMaDatabase} and,
between brackets, the ones from~\cite{Bollen18-1}.
In that list, we specify
the $62$ matroids that are not
algebraic over fields of characteristic $2$ because
they are not Frobenius-flock representable in characteristic $2$,
and the $10$ matroids that are not 
almost entropic because they have
a minor with the V\'amos configuration.
To the best of our knowledge, 
no result is known on 
the existence of algebraic representations
for the other 109 TTT matroids.
In addition, we found out that
the minors of $P_8$
that are discussed in Section~\ref{sec:8pointsRevisited}
appear as minors of some TTT matroids. 
All these facts are presented in
Table~\ref{tab:sparse pav TTT mats},
in which all 181 TTT matroids are listed.

In Figure~\ref{fig:allTTTpics}, the 181
TTT matroids, which are labeled with
the tags used in~\cite{Bollen18-1},
are arranged in columns 
according to the number of circuit-hyperplanes,
decreasing from left to right.
For a better viewing of this graph, 
the interested reader is invited to check \cite{Bam22}.
The edges indicate relaxation of
one circuit-hyperplane.
The red nodes are those that have been 
determined to be non-algebraic
over fields of characteristic $2$~\cite{Bollen18}, 
green ones are neither algebraic nor almost entropic 
due to a minor with the V\'amos configuration, 
while the
algebraic status of the blue ones is undetermined.
The first column contains
the only relaxation of $\ttc$ (with tag 45969),
which has $17$ circuit hyperplanes.  
The second column consists of the two relaxations
of that matroid together with $\twc$ 
(the green vertex with tag 94190).
In the last column we find only
the tic-tac-toe matroid $\ttt$ (tag 185672)
and in the last but one column we find
the matroid $\ttw$ (tag 152191),
the one of the second kind with
the fewest circuit-hyperplanes,
and some TTT matroids of the first kind.
 

\begin{figure}[!htp]
    \centering
    \includegraphics[width=1\linewidth]{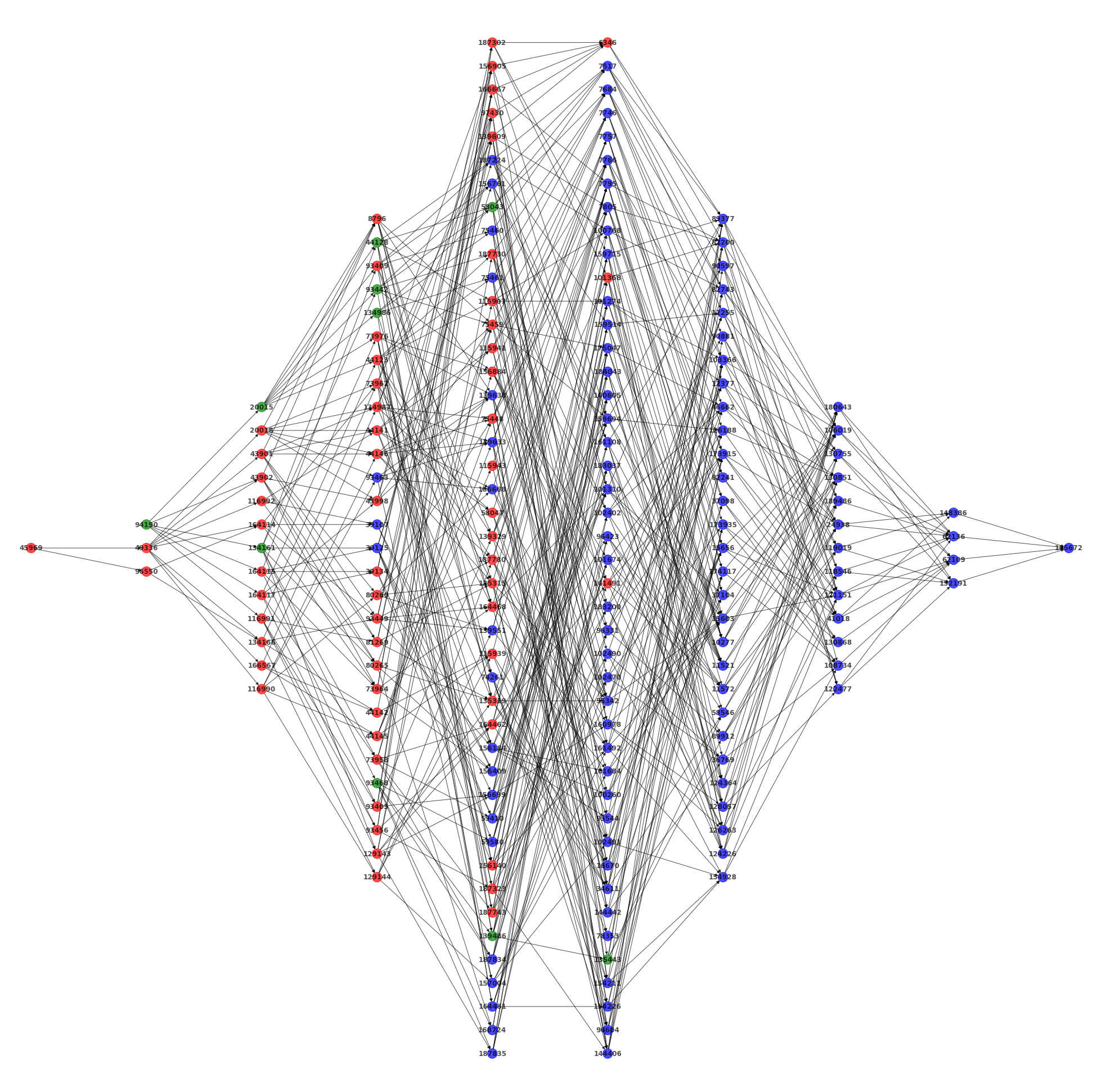}
    \caption{All TTT matroids on 9 points. The labels displayed are the Bollen tags~\cite{Bollen18-1} of these matroids}
    \label{fig:allTTTpics}
\end{figure}


\begin{longtable}[c]{@{}ccc@{}}
	\toprule
	Matroid &
	\begin{tabular}[l]{@{}c@{}}2-Algebraic\\ certificate\end{tabular} &
	Minor \\* \midrule
	\endfirsthead
	\multicolumn{3}{c}%
	{{\bfseries Table \thetable\ continued from previous page}} \\
	\toprule
	Matroid &
	\begin{tabular}[l]{@{}c@{}}2-Algebraic\\ certificate\end{tabular} &
	Minor \\* \midrule
	\endhead
	\bottomrule
	\endfoot
	\endlastfoot
& 17 circuit-hyperplanes  & \\   & & \\  
$264950$ $(45969)$  & 2-flock & $\pone$ 
\\  & & \\ \hline & 16 circuit-hyperplanes  & \\   & & \\
$\twc$ $264952$ $(94190)$  & I-M condition 
& $1509, \pone$ \\
$264955$ $(96550)$  & 2-flock & $\pone$ \\
$264956$ $(49336)$  & 2-flock & $\ptwo, \pone$ \\  
& & \\ \hline & 15 circuit-hyperplanes  & \\   & & \\
$264978$ $(164115)$ & 2-flock  & $\pone$ \\
$264984$ $(164117)$ & 2-flock  & $\ptwo, \pone$ \\
$264994$ $(20018)$  & 2-flock  & $\ptwo, \pone$  \\
$264999$ $(20015)$  & I-M condition &  $1509, \ptwo, \ptwop$ \\
$265008$ $(164114)$ & 2-flock  & $\ptwo, \pone$  \\
$265011$ $(134161)$ & I-M condition & $1509, \pone$ \\
$265012$ $(116992)$ & 2-flock & $\ptwo, \pone$ \\
$265014$ $(116990)$ & 2-flock  & $\pone$ \\
$265018$ $(116991)$ & 2-flock  & $\ptwo, \pone$ \\
$265020$ $(134166)$ & 2-flock  & $\ptwop, \pthree, \pone$ \\
$265023$ $(166567)$ & 2-flock & $\ptwop$ \\
$265026$ $(43901)$  & 2-flock  & $\ptwo, \ptwop, \pone$ \\
$265028$ $(43902)$  & 2-flock & $\pone$ \\
 & & \\ \hline & 14 circuit-hyperplanes  & \\   & & \\   
$265129$ $(129143)$ & 2-flock & $\ptwo, \pone$  \\
$265237$ $(93456)$  & 2-flock & $\ptwop, \pone$  \\
$265262$ $(73958)$  & 2-flock & $\ptwo, \ptwop, \pone$  \\
$265270$ $(39134)$  & 2-flock & $\ptwop, \pthree$ \\
$265389$ $(134987)$ & 2-flock & $\ptwop, \pone$  \\
$265409$ $(134986)$ & I-M condition  & $1509$ \\
$265420$ $(93442)$  & I-M condition  & $1509, \ptwo, \ptwop, \pone$ \\
$265421$ $(43998)$  & 2-flock  & $\ptwo, \pone$  \\
$265422$ $(73962)$  & 2-flock  & $\ptwo, \ptwop, \pthree, \pone$  \\
$265423$ $(44141)$  & 2-flock  & $\ptwo, \ptwop, \pone$ \\
$265424$ $(93463)$  & —    & $\ptwo$   \\
$265437$ $(44146)$  & 2-flock    & $\ptwop, \pone$  \\
$265451$ $(44128)$  & I-M condition    & $1509, \ptwo, \pone$ \\
$265465$ $(73976)$  & 2-flock & $\ptwop$ \\
$265468$ $(44123)$  & 2-flock & $\ptwop, \pone$ \\
$265547$ $(80265)$  & 2-flock  & $\pone$ \\
$265551$ $(80260)$  & 2-flock  & $\ptwop, \pone$  \\
$265552$ $(93449)$  & 2-flock  & $\pthree, \pone$   \\
$265553$ $(93405)$  & 2-flock  & $\pthree$  \\
$265555$ $(93409)$  & 2-flock  & $\pthree, \pone$  \\
$265556$ $(129144)$ & 2-flock  & $\ptwop, \pone$  \\
$265601$ $(39125)$  & —   & --- \\
$265602$ $(39107)$  & —   & $\ptwo$  \\
$265622$ $(44145)$  & 2-flock   & $\ptwop, \pone$  \\
$265623$ $(44142)$  & 2-flock & $\pone$      \\
$265695$ $(93468)$  & I-M condition & $1509$  \\
$265696$ $(81269)$  & 2-flock   & $\ptwo, \pone$  \\
$265715$ $(73964)$  & 2-flock   & $\pone$  \\
$265760$ $(8796)$   & 2-flock & $\ptwo, \ptwop$  \\  
& & \\ \hline & 13 circuit-hyperplanes  & \\   & & \\ 
$266399$ $(164462)$ & 2-flock  & $\pone$   \\ 
$266923$ $(187302)$ & 2-flock  & —   \\
$266948$ $(187323)$ & 2-flock  & $\ptwop$  \\
$267669$ $(75448)$  & 2-flock  & $\ptwop, \pone$    \\
$267671$ $(139633)$ & —                                                                 & $\ptwop$                                                                              \\
$267672$ $(139609)$ & 2-flock                                                            & $\pone$                                                                                \\
$267675$ $(58047)$  & 2-flock                                                            & $\pone$                                                                                \\
$267678$ $(97430)$  & 2-flock                                                            & $\ptwop, \pone$                                                                         \\
$267871$ $(75460)$  & —                                                                 & $\pone$                                                                                \\
$267873$ $(58043)$  & I-M condition   &  $1509, \pone$  \\
$267897$ $(164481)$ & —                                                                 & —                                                                                    \\
$267946$ $(75461)$  & —                                                                 & $\pone$                                                                                \\
$268016$ $(187730)$ & 2-flock                                                            & $\ptwop$                                                                              \\
$268017$ $(157004)$ & —                                                                 & $\ptwo, \ptwop$                                                                        \\
$268018$ $(166668)$ & —                                                                 & $\ptwo, \ptwop$                                                                        \\
$268099$ $(156140)$ & 2-flock                                                            & $\ptwop$                                                                              \\
$268115$ $(156884)$ & 2-flock                                                            & $\ptwop$                                                                              \\
$268120$ $(115907)$ & 2-flock                                                            & —                                                                                    \\
$268272$ $(59410)$  & —                                                                 & $\pthree$                                                                                \\
$268474$ $(187780)$ & 2-flock                                                            & $\ptwop$                                                                              \\
$268475$ $(166667)$ & 2-flock                                                            & $\pthree$                                                                                \\
$268476$ $(156184)$ & —                                                                 & $\ptwo, \pthree$                                                                          \\
$268477$ $(156699)$ & —                                                                 & $\pthree$                                                                                \\
$268486$ $(156761)$ & —                                                                 & $\pthree$                                                                                \\
$268611$ $(115943)$ & 2-flock                                                            & $\ptwop$                                                                              \\
$268613$ $(115939)$ & 2-flock                                                            & $\ptwop$                                                                              \\
$268765$ $(187743)$ & 2-flock                                                            & $\ptwop$                                                                              \\
$268774$ $(187224)$ & —                                                                 & $\pthree$                                                                                \\
$268805$ $(115941)$ & 2-flock                                                            & $\ptwop$                                                                              \\
$268958$ $(187834)$ & —                                                                 & $\ptwo$                                                                               \\
$268961$ $(187835)$ & —                                                                 & $\ptwo$                                                                               \\
$269060$ $(156409)$ & —                                                                 & —                                                                                    \\
$269061$ $(59580)$  & —                                                                 & —                                                                                    \\
$269062$ $(156905)$ & 2-flock                                                            & $\ptwop, \pthree$                                                                         \\
$269550$ $(75459)$  & 2-flock  & $\ptwo, \ptwop, \pone$ \\
$269551$ $(135315)$ & 2-flock                                                            & $\ptwop, \pone$                                                                         \\
$269557$ $(139486)$ & I-M condition   & $1509, \ptwo$ \\
$269558$ $(135319)$ & 2-flock                                                            & $\ptwo, \pone$                                                                          \\
$269559$ $(164468)$ & 2-flock                                                            & $\pthree, \pone$                                                                           \\
$269704$ $(139638)$ & —                                                                 & $\ptwop$                                                                              \\
$269824$ $(74261)$  & —                                                                 & —                                                                                    \\
$269895$ $(139329)$ & 2-flock                                                            & —                                                                                    \\
$270130$ $(168724)$ & —                                                                 & —                                                                                    \\
$270133$ $(139551)$ & —                                                                 & $\pone$                                                                                \\  
& & \\ \hline & 12 circuit-hyperplanes  & \\   & & \\
$273139$ $(101274)$ & —                                                                 & —                                                                                    \\
$273141$ $(159715)$ & —                                                                 & —                                                                                    \\
$273582$ $(96423)$  & —                                                                 & —                                                                                    \\
$274066$ $(161491)$ & 2-flock                                                            & $\ptwop$                                                                              \\
$274247$ $(184043)$ & —                                                                 & $\pthree$                                                                                \\
$275082$ $(183037)$ & —                                                                 & $\pthree$                                                                                \\
$275391$ $(7757)$   & —                                                                 & $\ptwop$                                                                              \\
$275394$ $(102491)$ & —                                                                 & $\ptwop$                                                                              \\
$275398$ $(7766)$   & —                                                                 & $\ptwop$                                                                              \\
$275399$ $(95544)$  & —                                                                 & —                                                                                    \\
$275410$ $(144406)$ & —                                                                 & $\ptwop$                                                                              \\
$275411$ $(6346)$   & 2-flock                                                            & —                                                                                    \\
$275416$ $(96331)$  & —                                                                 & $\ptwop$                                                                              \\
$275417$ $(7684)$   & —                                                                 & —                                                                                    \\
$276341$ $(7517)$   & —                                                                 & —                                                                                    \\
$276430$ $(102490)$ & —                                                                 & $\ptwop$                                                                              \\
$276671$ $(144442)$ & —                                                                 & —                                                                                    \\
$276792$ $(102470)$ & —                                                                 & $\ptwop$                                                                              \\
$277240$ $(183200)$ & —                                                                 & —                                                                                    \\
$277656$ $(78353)$  & —                                                                 & —                                                                                    \\
$277673$ $(100260)$ & —                                                                 & —                                                                                    \\
$280230$ $(55342)$  & —                                                                 & $\ptwop$                                                                              \\
$280241$ $(101674)$ & —                                                                 & —                                                                                    \\
$280246$ $(16670)$  & —                                                                 & —                                                                                    \\
$280249$ $(159694)$ & —                                                                 & $\pthree$                                                                                \\
$280253$ $(161108)$ & —                                                                 & $\pthree$                                                                                \\
$280254$ $(101310)$ & —                                                                 & $\pthree$                                                                                \\
$280733$ $(7746)$   & —                                                                 & $\ptwop$                                                                              \\
$280891$ $(102402)$ & —                                                                 & —                                                                                    \\
$281004$ $(161492)$ & —                                                                 & $\ptwop$                                                                              \\
$281568$ $(96604)$  & —                                                                 & $\pone$                                                                                \\
$281572$ $(34611)$  & —                                                                 & $\pone$                                                                                \\
$281574$ $(135443)$ & I-M condition    & $1509$  \\
$281581$ $(154226)$ & —                                                                 & —                                                                                    \\
$281794$ $(154211)$ & —                                                                 & $\pone$                                                                                \\
$282270$ $(101368)$ & 2-flock                                                            & —                                                                                    \\
$282271$ $(100768)$ & —                                                                 & —                                                                                    \\
$282272$ $(160978)$ & —                                                                 & —                                                                                    \\
$283581$ $(159514)$ & —                                                                 & —                                                                                    \\
$283624$ $(175047)$ & —                                                                 & $\ptwo$                                                                               \\
$283626$ $(160605)$ & —                                                                 & —                                                                                    \\
$283630$ $(7805)$   & —                                                                 & —                                                                                    \\
$283631$ $(7795)$   & —                                                                 & —                                                                                    \\
$283632$ $(161684)$ & —                                                                 & $\ptwo$                                                                               \\ 
 & & \\ \hline & 11 circuit-hyperplanes  & \\   & & \\
$291383$ $(124364)$ & —                                                                 & —                                                                                    \\
$292609$ $(126263)$ & —                                                                 & —                                                                                    \\
$293346$ $(82241)$  & —                                                                 & —                                                                                    \\
$293347$ $(15656)$  & —                                                                 & —                                                                                    \\
$293361$ $(82200)$  & —                                                                 & —                                                                                    \\
$294990$ $(82743)$  & —                                                                 & —                                                                                    \\
$295231$ $(12377)$  & —                                                                 & —                                                                                    \\
$299715$ $(15603)$  & —                                                                 & —                                                                                    \\
$299721$ $(11521)$  & —                                                                 & $\pthree$                                                                                \\
$300609$ $(11572)$  & —                                                                 & —                                                                                    \\
$300831$ $(126188)$ & —                                                                 & —                                                                                    \\
$301018$ $(128057)$ & —                                                                 & —                                                                                    \\
$303086$ $(103366)$ & —                                                                 & —                                                                                    \\
$303094$ $(58546)$  & —                                                                 & —                                                                                    \\
$303095$ $(89912)$  & —                                                                 & —                                                                                    \\
$303158$ $(173935)$ & —                                                                 & $\ptwop$                                                                              \\
$303165$ $(173915)$ & —                                                                 & $\ptwop$                                                                              \\
$303175$ $(90881)$  & —                                                                 & —                                                                                    \\
$304062$ $(174117)$ & —                                                                 & —                                                                                    \\
$304066$ $(89377)$  & —                                                                 & —                                                                                    \\
$304067$ $(90597)$  & —                                                                 & —                                                                                    \\
$304085$ $(154928)$ & —                                                                 & $\ptwop$                                                                              \\
$306452$ $(36769)$  & —                                                                 & —                                                                                    \\
$308279$ $(124226)$ & —                                                                 & —                                                                                    \\
$308280$ $(12255)$  & —                                                                 & —                                                                                    \\
$308285$ $(10277)$  & —                                                                 & —                                                                                    \\
$308381$ $(46662)$  & —                                                                 & —                                                                                    \\
$308385$ $(37104)$  & —                                                                 & —                                                                                    \\
$308386$ $(37098)$  & —     & —   \\ 
 & & \\ \hline & 10 circuit-hyperplanes  & \\   & & \\
$319504$ $(180643)$ & —     & —         \\
$320838$ $(130755)$ & —                                                                 & —                                                                                    \\
$327043$ $(24938)$  & —                                                                 & —                                                                                    \\
$327134$ $(41018)$  & —                                                                 & —                                                                                    \\
$327157$ $(119019)$ & —                                                                 & —                                                                                    \\
$328810$ $(106019)$ & —                                                                 & —                                                                                    \\
$328817$ $(189486)$ & —                                                                 & —                                                                                    \\
$328818$ $(118546)$ & —                                                                 & —                                                                                    \\
$328917$ $(108734)$ & —                                                                 & —                                                                                    \\
$328928$ $(121151)$ & —                                                                 & —                                                                                    \\
$328941$ $(122477)$ & —                                                                 & —                                                                                    \\
$335557$ $(130851)$ & —                                                                 & —                                                                                    \\
$335558$ $(130868)$ & —   & —   \\  
& & \\ \hline & 9 circuit-hyperplanes  & \\   & & \\
$350495$ $(148386)$ & —       & —    \\
$351377$ $(62136)$  & —     & —     \\
$351471$ $(62109)$  & —      & —    \\
$\ttw$ $351483$ $(152191)$ & —  & —  \\  
& & \\ \hline & 8 circuit-hyperplanes  & \\   & & \\
$\ttt$ $365084$ $(185672)$ & —  & ---   \\* \bottomrule
\caption{List of all 181 TTT matroids}
    \label{tab:sparse pav TTT mats}\\
\end{longtable}

\subsection{Other Matroids on Nine Elements}
\label{sec:exp}



The Ingleton inequality does not provide
any relevant example of matroid on $9$ elements
that is not linearly representable.
This is due to the fact that each matroid on $9$
elements violating the Ingleton inequality has a minor on $8$ elements 
with the same property~\cite{MaRo08}. 
In this section we present additional examples of 
rank $5$ matroids on $9$ elements that,
similarly to TTT matroids, satisfy the Ingleton inequality
but do no satisfy the common information property.

Those examples have been found by an extensive search
among the matroids of rank $5$ on $9$ elements 
in the database~\cite{RoMaDatabase}.
The ones that have a minor isomorphic to 
one of the 39 matroids on eight elements that do not 
satisfy Ingleton Inequality~\cite{MaRo08} were discarded. 
This procedure is most of the times faster than 
directly checking whether the Ingleton inequality is satisfied or not.
Then we searched for matroids that are not $1$-{\gep}
by checking the existence of the modular cuts 
associated to {\gep} extensions.
Finally, the existence of {\cip} extensions of those matroids
were checked for the pairs not admitting a {\gep} extension.
We note that this approach fails for matroids of rank 4
because the ones that are not $1$-{\gep} are precisely
those violating the Ingleton inequality~\cite{BaKe88}.

One example is the matroid $201827$ with flats 
$0125$, $0268$, and $1568$ of rank 3, 
and circuit-hyperplanes $12378$, $03458$, $24578$, $01467$, $12346$, and $34678$. 
By checking modular cuts corresponding to 
non-modular pair of flats, we see that
the pair $(035,146)$ does not admit a {\gep} extension.  
Then we checked that a {\cip} extension does not exist
for the same pair.

Our search produced the list of matroids in
Table~\ref{tab:ICnonSPnonGP}.
Together with TTT matroids, those are all  
Ingleton-compliant matroids of rank $5$ on $9$ 
elements that are not $1$-{\gep}.
Moreover, they are not $1$-{\cip}.
They are non sparse paving matroids.
Checking the duals of the matroids in that list, 
we found that they break {\akp} at depth 3 and are, therefore, neither almost entropic nor algebraic.

\begin{table}[ht]
	\centering
	\begin{tabular}{|l|l|l|l|l|l|}
		\hline
		199136 & 204630 & 211985 & 221647 & 227977 & 233153 \\ \hline
		199230 & 204769 & 216857 & 221650 & 228317 & 233156 \\ \hline
		199553 & 204896 & 217478 & 221722 & 228452 & 233245 \\ \hline
		199807 & 204903 & 217597 & 221834 & 229354 & 233261 \\ \hline
		200470 & 204973 & 217772 & 221905 & 229356 & 241344 \\ \hline
		200589 & 205074 & 217846 & 221910 & 229357 & 241614 \\ \hline
		200633 & 205111 & 218082 & 222035 & 229741 & 243049 \\ \hline
		200972 & 206383 & 218124 & 222041 & 229892 & 243792 \\ \hline
		201001 & 206385 & 218129 & 222044 & 230229 & 243800 \\ \hline
		201056 & 206515 & 218179 & 222384 & 230558 & 243801 \\ \hline
		201121 & 206844 & 218341 & 222385 & 231565 & 244422 \\ \hline
		201124 & 206959 & 220346 & 222436 & 231566 & 244449 \\ \hline
		201827 & 206992 & 220524 & 223015 & 231587 & 245708 \\ \hline
		201869 & 207536 & 220657 & 223016 & 231588 & 245732 \\ \hline
		201957 & 207550 & 221002 & 223035 & 231997 & 245765 \\ \hline
		201958 & 207669 & 221046 & 223221 & 232065 & 253254 \\ \hline
		202832 & 208093 & 221323 & 223417 & 232651 & 253828 \\ \hline
		204059 & 211135 & 221541 & 227084 & 232654 &        \\ \hline
		204585 & 211841 & 221542 & 227086 & 232824 &        \\ \hline
		204624 & 211983 & 221635 & 227789 & 233072 &        \\ \hline
	\end{tabular}
\caption{Matroids of rank $5$ on $9$ elements 
that are Ingleton-Compliant but are not $1$-{\gep},
identified with the tags from~\cite{RoMaDatabase}.}
	\label{tab:ICnonSPnonGP}
\end{table}



\begin{proposition}
	The matroids listed in Table~\ref{tab:ICnonSPnonGP} are Ingleton-compliant and non-1-{\cip}, and their duals are neither 3-{\cip} nor 3-{\akp}.
\end{proposition}

In terms of application to secret sharing, 
we also found bounds on the information ratio 
of linear secret sharing schemes for the ports 
of TTT matroids and matroids in this family. 
In the case of TTT matroids, the bound on the 
information ratio for linear secret sharing schemes is $6/5$. 
For some ports of $\ttt$, this is the exact value~\cite{BBFP20}. 
However, for the ports of the matroids in Table~\ref{tab:ICnonSPnonGP}, 
the behavior is less uniform. 
While the same bound $6/5$ applies in most of the cases,
for some of them the bound is $7/5$~\cite[Appendix B.3]{Bam21}. 


Finally, we present a third family of new 
non-linearly representable matroids. 
We found these matroids by exploring 
the list of 12,129 matroids
of rank $5$ on $9$ elements 
for which, according to~\cite{Bollen18},
the existence of Frobenius flock representations
in characteristic $2$ remains unsolved.
The matroids in Table~\ref{tab:fam2matroids} 
are 1-{\gep} but not 2-{\gep}, 
so they are not linearly representable. 
Then we checked the {\cip} property, 
and we found that they are not $2$-{\cip}. 
Therefore, matroids in this table are not folded-linear. 
The sets in the first column of the table are the 
combinations we found to cause the matroids 
in the columns to the right to be non-2-{\cip}. 
That is not an exhaustive list, 
there may be other matroids of rank $5$
on $9$ elements that are 1-{\gep} but not 2-{\gep}. 
Finally, we note that the duals of all those matroids 
have been shown to be non-algebraic 
due to failing Ingleton-Main at various depths~\cite{Bollen18,Bollen18-1}.

\begin{table}[ht]
\begin{center}
\begin{tabular}{|c|c|c|c|c|c|}
\hline
\begin{tabular}[c]{@{}c@{}}Sets\\ ($U_1,V_1$) \& ($U_2,V_2$)\end{tabular} & \multicolumn{5}{c|}{Matroids} \\ \hline
\multirow{11}{*}{(03, 127) \& (36, 087)} & 6182   & 89044  & 135432 & 156366 & 159246 \\
                                         & 6184   & 89045  & 103147 & 159112 & 164382 \\
                                         & 6206   & 89046  & 136699 & 159119 & 171945 \\
                                         & 6207   & 95441  & 140204 & 159184 & 171946 \\
                                         & 7493   & 95446  & 145710 & 159185 & 171954 \\
                                         & 8369   & 100735 & 146367 & 159202 & 171967 \\
                                         & 58384  & 100736 & 146368 & 159236 & 172039 \\
                                         & 59382  & 100755 & 147269 & 159241 & 172042 \\
                                         & 88940  & 100947 & 154115 & 159242 & 187245 \\
                                         & 88967  & 100983 & 156059 & 159245 & 187929 \\
                                         & 89019  & 100988 & 156062 &        &        \\ \hline
(57, 280) \& (05, 148)                   & 7848   & 7849   & 144432 &        &        \\ \hline
\multirow{10}{*}{(03, 128) \& (36, 087)} & 8080   & 15183  & 42484  & 100325 & 127743 \\
                                         & 10631  & 15189  & 42525  & 104169 & 146946 \\
                                         & 15107  & 26185  & 42526  & 106805 & 158280 \\
                                         & 15108  & 26186  & 77860  & 111823 & 174258 \\
                                         & 15115  & 26224  & 83155  & 111839 & 183860 \\
                                         & 15128  & 26227  & 83156  & 111871 & 183866 \\
                                         & 15129  & 37250  & 91265  & 127707 & 183877 \\
                                         & 15130  & 37253  & 96490  & 127710 & 190074 \\
                                         & 15140  & 37255  & 100320 & 127720 & 190075 \\
                                         & 15177  & 42465  & 100321 & 127742 &        \\ \hline
\multirow{2}{*}{(46, 018) \& (04, 135)}  & 8088   & 26584  & 83600  & 96491  & 184061 \\
                                         & 18212  &        &        &        &        \\ \hline
(01, 237) \& (15, 078)                   & 46632  & 55425  & 146335 &        &        \\ \hline
$(04, 128)$ \& $(02, 136)$                 & 156183 &        &        &        &        \\ \hline
\end{tabular}
\caption{Another list of 
non-{\cip} matroids,
identified with the tags from~\cite{Bollen18-1}.}
\label{tab:fam2matroids}
\end{center}
\end{table}
}


\begin{thebibliography}{11}

\bibitem{AhKo77}
Ahlswede, R., K{\"o}rner, J.:
On the connection between the entropies of input and output
distributions of discrete memoryless channels. 
Proceedings of the 5th Brasov Conference on
Probability Theory, Brasov, 1974.
Editura Academiei, Bucuresti, 13--23 (1977).
	
\bibitem{AhKo06}
Ahlswede, R., K{\"o}rner, J.:
Appendix: On Common Information and Related Characteristics of 
Correlated Information Sources.
\textit{General Theory of Information Transfer and Combinatorics.}
pp. 664--677. Springer, Berlin Heidelberg (2006).

\bibitem{AlHo95}
Alfter, M. and Hochst\"attler, W.: 
On pseudomodular matroids and adjoints.
Discrete Appl. Math. 60, 3--11 (1995).






\fullv{
\bibitem{BaKe88}
Bachem, A. and Kern, W.:
On sticky matroids.
Discrete Math. 69, 11--18 (1988).
}


\bibitem{BaWa89}
Bachem, A. and Wanka, A.:
Euclidean intersection properties.
Journal of Combinatorial Theory, Series B 47, 10--19 (1989).

\bibitem{Bam21}
Bamiloshin, M.:
Common Information Techniques for the Study of Matroid Representation and Secret Sharing Schemes.
\emph{PhD thesis}. (2021).

\bibitem{Bam22}
Bamiloshin, M.: 
TTT-And-Other-NonRepresentable-Matroids. GitHub (2021).
\url{https://github.com/bmilosh/TTT-And-Other-NonRepresentable-Matroids}.

\bibitem{BBFP20}
Bamiloshin, M., Ben-Efraim, A., Farr\`as, O., Padr\'o, C.:
Common Information, Matroid Representation, and Secret Sharing for Matroid Ports.
Des. Codes Cryptogr. 89, 143--166 (2021).

\confv{
\bibitem{BFP23}
Bamiloshin, M., Farr\`as, O., Padr\'o, C.:
A Note on Extension Properties and Representations of Matroids.
arXiv.org,  arXiv:2306.15085 (2023)
}









\bibitem{BLP08}
Beimel, A., Livne, N.,  Padr\'o, C.:
Matroids Can Be Far From Ideal Secret Sharing.
\textit{Fifth Theory of Cryptography Conference, TCC 2008,
Lecture Notes in Comput.\ Sci.}
{\bfseries 4948} (2008) 194--212.
	

\bibitem{BeOr11}
Beimel, A., Orlov, I.:
Secret Sharing and Non-Shannon Information Inequalities.
IEEE Trans. Inform. Theory
57, 5634--5649 (2011).
	
\bibitem{Ben16}
Ben-Efraim, A.:
Secret-sharing matroids need not be algebraic.
Discrete Math., 339, 2136--2145 (2016).


\bibitem{Bollen18}
Bollen, G.P.:
Frobenius flocks and algebraicity of matroids. 
Eindhoven: Technische Universiteit Eindhoven.
\emph{PhD thesis}. (2018).

\bibitem{Bollen18-1}
Bollen, G.P.: 
Algebraicity of Matroids and Frobenius Flocks. 
GitHub (2018).
\url{https://github.com/gpbollen/Algebraicity-of-Matroids-and-Frobenius-Flocks}

\bibitem{BCD18}
Bollen, G.P., Dustin Cartwright, D., Draisma, J.:
Matroids over one-dimensional groups.
International Mathematics Research Notices, Volume 2022, Issue 3, February 2022, Pages 2298–2336.
Also available at arXiv:1812.08692 [math.CO] (2018).

\bibitem{BrDa91}
Brickell, E.F., Davenport, D.M.:
On the Classification of Ideal Secret Sharing Schemes. 
J. Cryptology,  4, 123--134 (1991).



\bibitem{Csi19}
Csirmaz, L.: Secret sharing and duality.
J. Math. Cryptol., 15, 157--173 (2021).



\bibitem{Dij97}
van Dijk, M.: 
More information theoretical inequalities to be used
in secret sharing?  
Inform. Process. Lett. 63, 41--44 (1997).

\bibitem{DFZ06}
Dougherty, R., Freiling, C., Zeger, K.:
Six new non-Shannon information inequalities.
In: 2006 IEEE International Symposium on Information Theory,
pp. 233--236 (2006).
	
\bibitem{DFZ07}
Dougherty, R., Freiling, C., Zeger, K.:
Networks, matroids, and non-Shannon information inequalities.
IEEE Trans. Inform. Theory 53 (2007), no. 6, 1949--1969. 

\bibitem{DFZ09}
Dougherty, R., Freiling, C., Zeger, K.:
Linear rank inequalities on five or more variables.
arXiv.org, arXiv:0910.0284v3 (2009).

\bibitem{DFZ11}
Dougherty, R., Freiling, C., Zeger, K.:
Non-Shannon Information Inequalities in Four Random Variables.
arXiv.org, arXiv:1104.3602v1 (2011).


\bibitem{DrLo87} 
Dress, A. and Lov\'asz, L.:
On some combinatorial properties of algebraic matroids.
Combinatorica 7 (1) 39--48 (1987).



\bibitem{FKMP20}
Farr\`as, O., Kaced, T., Mart\'{\i}n, S., Padr\'o, C.:
Improving the Linear Programming Technique in the Search
for Lower Bounds in Secret Sharing.
IEEE Trans. Inf. Theory 66(11): 7088--7100 (2020).


\bibitem{Fuj78}
Fujishige, S.:
Polymatroidal Dependence Structure
of a Set of Random Variables.
Information and Control 39, 55--72 (1978).
	
\bibitem{Fuj78-2}
Fujishige, S.:
Entropy functions and polymatroids---combinatorial
structures in information theory.
Electron. Comm. Japan 61, 14--18 (1978).


\bibitem{GOVW00}
Geelen, J.F., Oxley J.G., Vertigan, D.L., Whittle, G.P.: 
On the Excluded Minors for Quaternary Matroids.
Journal of Combinatorial Theory, Series B 80, 57--68 (2000).




\bibitem{GuRo19}
G\"urpinar, E., Romashchenko, A.:
How to Use Undiscovered Information Inequalities:
Direct Applications of the Copy Lemma.
Available at arXiv:1901.07476v2 (2019).

\bibitem{HRSV00}
Hammer, D.,  Romashchenko, A.E.,
Shen, A., Vereshchagin, N.K.:
Inequalities for Shannon entropy and Kolmogorov complexity.
Journal of Computer and Systems Sciences
60, 442--464 (2000).


\bibitem{Hoc97}
Hochst\"attler, W.:
About the Tic-Tac-Toe matroid. Technical Report 97.272, 
Universit\"at zu K\"oln, Angewandte Mathematik und Informatik (1997).

\bibitem{Ing71}
Ingleton, A.W.:
Representation of matroids.
In: \textit{Combinatorial Mathematics and its Applications\/}, D.J.A Welsh (ed.), pp.\ 149--167. Academic Press, London (1971).

\bibitem{InMa75}
Ingleton, A. W. and Main, R. A.:
Non-Algebraic Matroids Exist. {\em Bull. London Math. Soc.,} 7, 144--146, (1975).


\bibitem{JaMa96}
Jackson, W.A.,  Martin, K.M.:
Perfect secret sharing schemes
on five participants.
Des. Codes Cryptogr. 9, 267--286 (1996).

\bibitem{Kac13}
Kaced, T.:
Equivalence of Two Proof Techniques
for Non-Shannon Inequalities.
arXiv:1302.2994 (2013).

\bibitem{Kac18}
Kaced, T.:
Information inequalities are not closed under polymatroid duality. 
IEEE Trans. Inform. Theory 64, 4379--4381 (2018).

\bibitem{Lindstrom1986}
Lindstr{\"o}m, B.:
A Non-Linear Algebraic Matroid with Infinite Characteristic Set.
Discrete Math., 59, 319--320 (1986).

\bibitem{Lindstrom1988}
Lindstr\"om, B.:
A Generalization of the Ingleton-Main Lemma and A Class of Non-Algebraic Matroids.
Combinatorica 8 (1) 87--90, (1988).




\bibitem{MPY16}
Mart\'{\i}n, S., Padr\'o, C., Yang, A.:
Secret sharing, rank inequalities, and information inequalities. 
IEEE Trans. Inform. Theory 62, 599--609 (2016).



\bibitem{Mat99}
Mat\'u\v{s}, F.:
Matroid representations by partitions.
Discrete Math.
\textbf{203}, 169--194 (1999).

\bibitem{Mat07-3}
Mat\'u\v{s}, F.:
Infinitely many information inequalities.
In: \textit{Proc. IEEE International
Symposium on Information Theory, (ISIT)\/},
pp. 2101--2105 (2007).



\bibitem{Matus24}
Mat\'u\v{s}, F.:
Algebraic matroids are almost entropic.
{\em Proceedings of the American Mathematical Society}
152, 1--6 (2024).


\bibitem{MaRo08}
Mayhew, D., Royle, G.F.:
Matroids with nine elements.
J. Combin. Theory Ser. B 98, 415--431 (2008).






\bibitem{Oxley11}
Oxley, J.G:
Matroid theory. Second edition.
Oxford Science Publications,
The Clarendon Press, Oxford University Press,
New York (2011).


\bibitem{PVY13}
Padr\'o, C., V\'azquez, L., Yang, A.:
Finding Lower Bounds on the Complexity
of Secret Sharing Schemes by Linear Programming.
Discrete Appl. Math.
161,  1072--1084 (2013).




\bibitem{RoMaDatabase}
Royle, G., Mayhew, D.:
Matroids on 9 elements.
UWA Research Repository. 
DOI:10.26182/5e3378f0ca2cd



\bibitem{TCG17}
Thakor, S., Chan, T., Grant, A.:
Capacity bounds for networks with correlated sources 
and characterisation of distributions by entropies.
IEEE Trans. Inform. Theory
63, 3540--3553 (2017).
	


	
\bibitem{Yeu08}
Yeung, R.W.:
Information theory and network coding. 
Springer (2008).

\bibitem{Yeung12}
Yeung, R. W.:
A first course in information theory.
Springer Science \& Business Media (2012).

\bibitem{ZhYe98}
Zhang, Z., Yeung, R.W.:
On characterization of entropy function via
information inequalities.
IEEE Trans. Inform. Theory
44, 1440--1452 (1998).

\end{thebibliography}
\end{document}